\documentclass[a4paper]{amsart}

\usepackage[a4paper,width={16cm},left=2.5cm,bottom=3cm, top=3cm]{geometry}
\usepackage{enumerate}

\usepackage[utf8]{inputenc}
\usepackage[english]{babel}

\usepackage{a4wide}
\usepackage{xcolor}
\usepackage{amsmath}
\usepackage{amssymb}
\usepackage{amsfonts}
\usepackage{amsthm,graphicx}

\usepackage{hyperref}
\usepackage{mathtools}
\mathtoolsset{showonlyrefs}

\newtheorem{theorem}{Theorem}[section]
\newtheorem{lemma}[theorem]{Lemma}
\newtheorem{corollary}[theorem]{Corollary}

\newtheorem{proposition}[theorem]{Proposition}
\theoremstyle{definition}
\newtheorem{definition} [theorem] {Definition}

\newtheorem{remark} [theorem] {Remark}

\newcommand{\mc}{\mathcal}
\newcommand{\mb}{\mathbb}

\newcommand{\la}{\lambda}
\newcommand{\norm}[1]{\left\lVert#1\right\rVert}
\newcommand{\pd}[2]{\frac{\partial#1}{\partial#2}}

\newcommand{\R}{\mb{R}}

\newcommand{\e}{\varepsilon}

\newcommand{\Tr}{\mathop{\rm{Tr}}}
\newcommand{\dive}{\mathop{\rm{div}}}

\newcommand{\df}[4]{\ensuremath\sideset{_{#1}}{_{#4}}{\mathop{\left\langle #2, #3 \right\rangle}}}
\newcommand{\ps}[3]{\left( #2, #3 \right)_{#1}}
\newcommand{\con}{{\rm{const}\,}}

\newcommand{\iy}{\infty}

\newcommand{\vp}{\varphi}

\newcommand{\dyle}{\displaystyle}

\newcommand{\n}{\nabla}

\begin{document}

\title[ Fractional heat equation involving Hardy-Leray Potential]
{Fujita exponent for \textbf{non-local} parabolic equation involving the
Hardy-Leray potential}
\thanks{Work partially supported by the MUR-PRIN-20227HX33Z “Pattern formation in nonlinear phenomena”, the Project PDI2019-110712GB-100, MINECO, Spain and "Ayuda de Excelencia al profesorado universitario", in UAM.
The first author is also partially supported by an Erasmus grant
from Autonoma University of Madrid.}
\author[B. Abdellaoui, G. Siclari, A. Primo ]{Boumediene Abdellaoui, Giovanni Siclari and Ana Primo}

\address{\hbox{\parbox{5.7in}{\medskip\noindent {B. Abdellaoui, Laboratoire d'Analyse Nonlin\'eaire et Math\'ematiques
Appliqu\'ees. \hfill \break\indent D\'epartement de
Math\'ematiques, Universit\'e Abou Bakr Belka\"{\i}d, Tlemcen,
\hfill\break\indent Tlemcen 13000, Algeria.}}}}

\address{\hbox{\parbox{5.7in}{\medskip\noindent {G. Siclari, Dipartimento di Matematica "Franceso Brioschi"
\hfill \break\indent
Politecnico di Milano
\hfill\break\indent
Via Giuseppe Ponzio 31, 20133, Milano, Italy}}}}

\address{\hbox{\parbox{5.7in}{\medskip\noindent{A. Primo, Departamento de Matem\'aticas,\\ Universidad Aut\'onoma de Madrid,\\
        28049, Madrid, Spain. \\[3pt]
        \em{E-mail addresses: }\\{\tt boumediene.abdellaoui@inv.uam.es, \tt giovanni.siclari@polimi.it, ana.primo@uam.es
         }.}}}}

\date{\today}


 \begin{abstract}
In this paper we analyse the existence and non-existence of non-negative solutions to a non-local parabolic equation with a Hardy-Leray type potential.
More precisely, we consider the problem
\begin{equation*}
\begin{cases}
(w_t-\Delta w)^s=\frac{\la}{|x|^{2s}} w+w^p +f, &\text{  in  }\R^N\times (0,+\infty),\\
w(x,t)=0, &\text{  in  }\R^N\times (-\infty,0],
\end{cases}
\end{equation*}
where $N> 2s$, $0<s<1$ and $0<\la<\Lambda_{N,s}$, the optimal
constant in the fractional  Hardy-Leray inequality.  In particular we show the existence of a critical existence exponent $p_{+}(\lambda, s)$ and of a Fujita-type exponent $F(\la,s)$ such that the following holds:
\begin{itemize}
\item  Let  $p>p_+(\la,s)$. Then there are not any   non-negative  supersolutions.
\item Let $p<p_+(\la,s)$.Then there exist  local solutions while concerning global solutions we need to distinguish two cases:
\begin{itemize}
\item Let $ 1< p\le F(\la,s)$. Here we show that a
weighted norm of any positive solution blows up in finite time.
\item Let $F(\la,s)<p<p_+(\la,s)$. Here we prove the existence of
global solutions under suitable hypotheses.
\end{itemize}
\end{itemize}

\end{abstract}

\maketitle

\section{Introduction}\label{sec_introduction}
The main goal of this work is to analyse the existence or non-existence of positive global solutions to the non-local parabolic equation given by
\begin{equation}\label{prob_frac_heat_hardy}
\begin{cases}
(w_t-\Delta w)^s=\dfrac{\la}{|x|^{2s}} w+w^p+f, &\quad \text{ in } \R^N \times (0,\infty),\\
w(x,t)=0 &\quad \text{ in } \R^N \times (-\infty,0],
\end{cases}
\end{equation}
where $ s \in (0,1)$, $N\ge2$,  $\la \in (0,\Lambda_{N,s})$ (see \eqref{def_La}), $p>1$, and $f$ is a
non-negative function satisfying suitable integrability hypothesis.

The  operator $H^s (w):=(w_t-\Delta w)^s$ is the fractional heat
operator. In the literature there are many possible ways to define
$H^s$, see for example \cite{AT} for an overview. Among the papers
concerning the operators $H^s$ and its properties we cite
\cite{ST} for regularity results, a Harnack inequality, an
extension result and a monotonicity formula for solution to the
equation $ H^s (w)=0$. The more general case of smooth potentials
has been treated by \cite{BG}, where the authors obtain a
monotonicity formula by the means of a suitable H\"older
regularity results and an extension theorem.

In the presence of an Hardy-type potential, a  monotonicity
formula and the asymptotic behaviour near the
singular point $0$ was investigated in \cite{FPS_unique}. A
detailed analysis of local solutions to the extended problem has
been  carried out in \cite{AT}. Furthermore, a more general
extension theorem for a vast class of parabolic fractional
operators, including $H^s$, as been obtained in \cite{BCS}. We
will give a  precise definition of the operator $H^s$  by means of
the  Fourier transform in  Section \ref{sec_functional_settings}.

The study of existence of global solutions to problem
\eqref{prob_frac_heat_hardy}, is related to the study of the  so
called Fujita-type exponent, see \cite{F}. Furthermore the
presence of an Hardy-type potential obstructs  even the existence
of local solutions for large $p$. Before stating our main results,
let us make a brief overview of the literature concerning Fujita
exponents and  non-existence results when there is a Hardy-type
potential.

In the semi-linear case, $s=1$ and $\la=0$,  Fujita proved, in his
fundamental work \cite{F}, the existence of a critical
exponent $F=1+\frac{2}{N}$ such that if $1<p<F$, then any solution
to the semi-linear problem
\begin{equation}\label{Fujita}\left\{
\begin{array}{l}
u_t=\Delta u+u^p,\, x\in\R^n,\, t>0,\\ u(x,0)=u_0(x)\geq
0,\,x\in\R^n,
\end{array}\right.
\end{equation}
blows up in finite time. The critical case $p=F$ was considered
later in \cite{W}. It is proved that in that case, a suitable norm
of the solution goes to infinity in a finite time.

In the case of fractional diffusion, the author in \cite{SS}
considered the equation
\begin{equation}\label{nolocal}
\left\{
\begin{array}{rcll}
 u_t+(-\Delta)^{s} u&=& u^p &\text{ in } \R^N\times (0,T),\\
u(x,0)&=& u_0(x) &\text{ if }x\in\R^N,
\end{array}
\right.
\end{equation}
where $N> 2s$ and $0<s<1$.  Notice that $(-\Delta)^s$ is the
fractional Laplacian defined by
\begin{equation}\label{fraccionario}
(-\Delta)^{s}u(x):=a_{N,s}\text{ P.V.}\int_{\mathbb{R}^{N}}{\frac{u(x)-u(y)}{|x-y|^{N+2s}}\, dy},\,s\in(0,1),
\end{equation}
where
$$a_{N,s}=2^{2s-1}\pi^{-\frac N2}\frac{\Gamma(\frac{N+2s}{2})}{|\Gamma(-s)|}$$
is the normalization constant. See \cite{DPV},
\cite{FLS}, \cite{La_book} and the references therein for additional properties of
the fractional Laplacian. In this case the Fujita critical
exponent takes the form $F(s)=1+\frac{2s}{N}$.

Now, in the case $\la>0$, the problem is related to the following Hardy-Leray type
inequality:
\begin{theorem}\label{DH}
For all $\varphi\in \mathcal{C}^{\infty}_{0}(\R^n)$, we have
\begin{equation}\label{ineq_Hardy_frac_Herbst}
\Lambda_{N,s}\,\int_{\R^N} \frac{\varphi^2(x)}{|x|^{2s}}\,dx\le
\int_{\R^N}\int_{\R^N}
\frac{(\varphi(x)-\varphi(y))^2}{|x-y|^{N+2s}}dxdy,
\end{equation}
with
\begin{equation}\label{def_La}
\Lambda_{N,s}=2^{2s}\dfrac{\Gamma^2(\frac{N+2s}{4})}{\Gamma^2(\frac{N-2s}{4})}.
\end{equation}
The constant $\Lambda_{N,s}$ is optimal and is not attained.
Moreover, $\Lambda_{N,s}\to
\Lambda_{N,1}:=\left(\dfrac{N-2}{2}\right)^2$, the classical Hardy
constant, when $s$ tends to $1$.
\end{theorem}
This inequality was first proved in  \cite{He_spectral}. See also
\cite{PS_book} for a direct proof.

The case $s=1$ and $\la>0$ was treated in \cite{APP}. More
precisely, if we consider the problem
\begin{equation}\label{Cauchyapp}
\begin{cases}
u_t-\Delta u=\lambda\dfrac{u}{|x|^{2}}+u^{p}, &\text{ in }\R^N\times(0,\infty),\\
u(x,0)=u_{0}(x), &\text{ in }\R^N,
\end{cases}
\end{equation}
setting
$\mu(\la)=\frac{N-2}{2}-\sqrt{\Big(\frac{N-2}{2}\Big)^2-\lambda}$,
then if $p\ge p_+(\la):=1+\frac{2}{\mu(\la)}$ there are no local positive supersolutions.

Due to the presence of the Hardy potential, then
it is possible to show that any positive solution of
\eqref{Cauchyapp} satisfies
$$
u(x,t)\ge C(t,r)|x|^{-\mu(\la)} \mbox{  in   }B_r(0), t>0.
$$
See for instance \cite{APP}. 

Furthermore  in the spirit of Fujita blow-up
exponent, it is proved in \cite{APP} that if
$1<p<1+\frac{2}{N-\mu(\la)}$, there exists $T^*>0$, independent
of the initial datum, such that the solution $u$ to equation
\eqref{Cauchyapp} satisfies
\begin{equation}\label{blow-up-local}
\lim\limits_{t\to
T^{*}}\int_{B_{r}(0)}|x|^{-\mu(\la)}u(x,t)\,dx=\iy,
\end{equation}
for any ball $B_{r}(0)$.  Finally if
$1+\dfrac{2}{N-\mu(\la)}<p<1+\dfrac{2}{\mu(\la)}$, if the initial
datum is small enough, there exists a global solution to
\eqref{Cauchyapp}.

Hence we can define the Fujita type exponent in this case by
setting $F(\lambda)= 1+\dfrac{2}{N-\mu(\lambda)}$.  Under
fractional diffusion the equation \eqref{Cauchyapp} takes the form
\begin{equation}\label{Cauchys}
\begin{cases}
    u_t+(-\Delta)^s u=\lambda\dfrac{u}{|x|^{2}}+u^{p}, &\text{ in }\R^N\times(0,\infty),\\
    u(x,0)=u_{0}(x), &\text{ in }\R^N.
\end{cases}
\end{equation}
In this case, it was proved in \cite{APP_Fujita} that
 if $p\ge
\tilde{p}_+(\la,s):=1+\frac{2s}{\mu(\la)}$, there are no local
positive supersolutions, while the Fujita critical exponent is
given by $\tilde{F}(\lambda,s)= 1+\dfrac{2s}{N-\mu(\lambda)}$
where
\begin{equation}\label{def_mu_la}
\mu(\lambda)= \dfrac{N-2s}{2}-\alpha_{\lambda}
\end{equation}
and
$\alpha_{\lambda}$ is defined by the implicit formula
$$
\Upsilon(\alpha_{\lambda})=\dfrac{2^{2s}\,\Gamma(\frac{N+2s+2\alpha_{\lambda}}{4})\Gamma(\frac{N+2s-2\alpha_{\lambda}}{4})}{\Gamma(\frac{N-2s+2\alpha_{\lambda}}{4})\Gamma(\frac{N-2s-2\alpha_{\lambda}}{4})},
$$
see also Proposition \ref{prop_mu}.

Going back to the operator $H^s$, in the case $\la=0$, the author in \cite{T_pointwise,T_frac_heat} considers the problem
\begin{equation}\label{one}
\begin{cases}
0\le(w_t-\Delta w)^s\le w^p, &\quad \text{ in } \R^N \times (0,\infty),\\
w(x,t)=0, &\quad \text{ in } \R^N \times (-\infty,0].
\end{cases}
\end{equation}
Using an integral representation formula for the inverse of $H^s$ he was able to show the existence of a
critical Fujita exponent given by $F(s)=1+\frac{2s}{N+2-2s}$. The proof is deeply based on some convolution properties of Gaussian
functions.

The situation when $\la>0$ is more delicate. The argument used in \cite{T_frac_heat} seems to be directly not applicable in our
situation. Our approach is instead based on the extension result   proved   for fractional powers of the heat operator  in \cite{BCS}.

Let us define the critical non-existence exponent $p_+(\la,s)$ and the Fujita exponent $F(\la,s)$ respectively as:
\begin{equation}\label{def_p+las_Flas}
p_+(\la,s):=1+\frac{2s}{\mu(\lambda)} \quad \text{ and } F(\la,s):=1+\frac{2s}{N+2-2s-\mu(\la)},
\end{equation}
with $\mu(\la)$ as in \eqref{def_mu_la}.  We distinguish  different cases according to the value of  the exponent $p$. Our main results are the following.
\begin{itemize}
\item  The case  $p>p_+(\la,s)$. Here, as in the elliptic case,
using the local behaviour of the solution near the origin obtained
in Proposition \ref{prop_estimate_singularity}, and thanks to a
Picone-type inequality (see Proposition \ref{prop_Picone} below),
we are able to show that there are not any weak non-trivial
solution  to problem \eqref{prob_frac_heat_hardy}.  In particular
if $f\not \equiv 0$ there are not any non-negative weak
supersolutions. \item The case $p<p_+(\la,s)$. Here it is possible
to show the existence of local solutions, see Remark
\ref{reamrk_local}. On the other hand, concerning global solutions
we distinguish two cases:
\begin{itemize}
\item The case $F(\la,s)<p<p_+(\la,s)$. Here we prove the
existence of a very weak supersolution. Thanks to
Proposition \ref{prop_existence_very_weak_solution}, we get the
existence of a very weak solution to equation
\eqref{prob_frac_heat_hardy}. \item The case $ 1< p\le F(\la,s)$.
Here we show that any positive weak or very weak solution of
\eqref{prob_frac_heat_hardy}  blows up, in a suitable sense, in
finite time. See Theorem \ref{theor_Fujita}.
 \end{itemize}
\end{itemize}

We remark that our result are coherent with \cite{T_frac_heat}. Indeed $\mu(\la) \to 0^+$, $p_+(\la,s)\to \infty$ and $F(\la,s)\to F(s)$ as $\la \to 0^+$, where $F(s)$ is the Fujita exponent of problem \eqref{one}. Furthermore the technique that we  used  to
prove  Theorem \ref{theor_Fujita} can be easily adapted to the case $\la=0$,  see Remark \ref{remark_la=0} and also  \cite{T_frac_heat}.

It is also interesting to notice that the critical non-existence
exponent $p_+(\la,s)$ is the same for problems \eqref{Cauchys} and
\eqref{prob_frac_heat_hardy} and that $p_+(\la,s) \to p_+(\la)$ as
$s \to 1^+$, where $ p_+(\la)$ is  the critical non-existence exponent
of problem \eqref{Cauchyapp}. This is due to the fact the exponent
$p_+(\la,s)$ is determined only by the elliptic part of the
operator and the presence of the Hardy-type
potential, which induces a singular behaviour near $0$ for any
non-negative supersolutions.

On the other hand the Fujita exponents $F(\la,s)$ and
$\tilde{F}(\la,s)$ of problems  \eqref{Cauchys}  and
\eqref{prob_frac_heat_hardy} are different and in particular
$F(\la,s) <\tilde{F}(\la,s)$ for any $s \in (0,1)$ and $\la \in
(0,\Lambda_{N,s})$. Nevertheless we have that  $F(\la,s) \to
F(\la)$ and $\tilde{F}(\la,s) \to F(\la)$  as $s \to 1^-$,  where
$F(\la)$  is the Fujita exponent of  \eqref{Cauchyapp}.

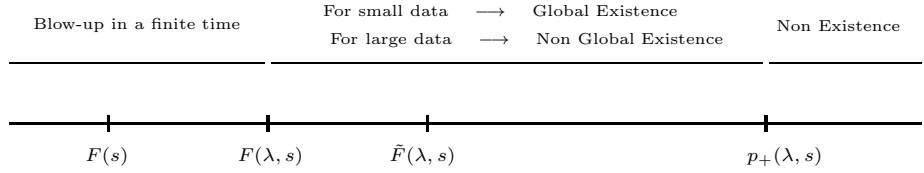
\begin{figure}
	\begin{center}
		\setlength{\unitlength}{1mm}
		\begin{picture}(120,20)
			\put(0,10){\line(1,0){33.5}}
			\put(34.5,10){\line(1,0){64.5}}
			\put(100,10){\line(1,0){20}}
			\thicklines
			\put(0,2){\line(1,0){120}}
			\put(13,1){\line(0,1){2}}
			\put(34,1){\line(0,1){2}}
			\put(55,1){\line(0,1){2}}
			\put(99.5,1){\line(0,1){2}}
			\put(16.75,15){\makebox(0,0){\tiny Blow-up in a finite time}}
			\put(65,17){\makebox(0,0){\tiny For small data
					$\quad\longrightarrow\quad$ Global Existence }}
			\put(68,13){\makebox(0,0){\tiny For large data
					$\quad\longrightarrow\quad$ Non Global Existence}}
			\put(109,15){\makebox(0,0){\tiny Non Existence}}
			\put(10,-3){\scriptsize $F(s)$}
			\put(30,-3){\scriptsize $F(\la,s)$}
			\put(50,-3){\scriptsize $\tilde{F}(\la,s)$}
			\put(97,-3){\scriptsize $p_+(\lambda,s)$}
		\end{picture}
	\end{center}
\caption{Fujita exponent for the fractional heat equation with Hardy potential.}
\end{figure}

This paper is organized as follows. In the Section
\ref{sec_functional_settings_and_preliminaries} we introduce some
definitions and useful tools about the fractional power of the
heat equation, the sense for which solution are defined, some
representation formulas and preliminary results.

In Section \ref{section_ground_state} we consider the operator
$H^s$ perturbed by the Hardy potential and prove  ground state
representation formula for $H^s(\cdot)-\la\frac{\cdot}{|x|^{2s}}$.
In Section \ref{section_comparison} we obtain a comparison
principle for weak solutions using the extension procedure
developed in \cite{BCS}. As a consequence, we show the existence
of weak solutions in a suitable sense of problem
\ref{prob_frac_heat_hardy}, if there exists a supersolution of the
same problem. In Section \ref{sec_local_weak_sol}, we analyse the
local behaviour of weak supersolutions. In particular  we are able
to obtain pointwise estimates from below and a strong maximum
principle as consequence.

Finally our main results are proved in Section
\ref{section_main_results}, while in Section
\ref{sec_open_problems}, we discuss some remaining  open problems
and possible further developments.

\section{Preliminaries}\label{sec_functional_settings_and_preliminaries}
\subsection{Functional setting} \label{sec_functional_settings}
For any real Hilbert space $X$ we denote with $X^*$ its
dual space and with $\df{X^*}{\cdot}{\cdot}{X}$ the duality between $X^*$ and $X$; $\ps{X}{\cdot}{\cdot}$ denotes the scalar product in $X$.

We define the operator  $H^s$  by  means of the Fourier transform as follows:
\begin{equation}\label{def_Hs_fourier}
\widehat{ H^s (w)}(\xi,\theta):=(i\theta+|\xi|^2)^s\widehat{w}(\xi,\theta),
\end{equation}
where the Fourier transform of $w$ is defined as
\begin{equation*}
\mathcal{F}(w)(\xi,\theta)=\widehat{w}(\xi,\theta):=\int_{\R^{N+1}}e^{-i(x \cdot\xi+t \theta )}w(x,t) \, dx\,dt.
\end{equation*}
Then the natural domain for a pointwise definition of $H^s$ is the set
\begin{equation*}
\left\{w \in L^2(\R^{N+1}):|i\theta+|\xi|^2|^s\widehat{w} \in L^2(\R^{N+1})\right\}.
\end{equation*}
Furthermore, if $\phi \in \mc{S}(\R^N)$, then by  \cite[Theorem
1.1]{ST}, we have the following pointwise formula for   the
operator $H^s$
\begin{equation}\label{def_Hs_pointwise}
H^s\phi(x,t)=
\frac{1}{(4\pi)^{\frac{N}{2}}\Gamma(-s)}\int_{-\infty}^t\int_{\R^N}\frac{\phi(x,t)-\phi(y,\sigma)}{(t-\sigma)^{\frac{N}{2}+s+1}}
e^{-\frac{|x-y|^2}{4(t-\sigma)}}dyd\sigma.
\end{equation}
Note that we have stated the previous formula for $\phi \in \mc{S}(\R^N)$ for the sake of simplicity,
see \cite[Theorem 1.1]{ST} for less restrictive regularity assumption on $\phi$.

We can  extend $H^s$  on
\begin{equation*}
\mathop{\rm{Dom}}(H^s):=\left\{w \in L^2(\R^{N+1}):\int_{\R^{N+1}}|i\theta+|\xi|^2|^s|\widehat{w}(\xi,\theta)|^2 \,d\xi\, d\theta<+\infty\right\},
\end{equation*}
endowed with the norm
\begin{equation*}
\norm{w}_{\mathop{\rm{Dom}}(H^s)}:=\left(\int_{\R^{N+1}} w^2(x,t) \,dx\,dt\right)^{\frac12}+\left(\int_{\R^{N+1}}|i\theta+|\xi|^2|^{s}|\widehat{w}(\xi,\theta)|^2d\xi \,d\theta\right)^{\frac12},
\end{equation*}
as the map from $\mathop{\rm{Dom}}(H^s)$ into its dual space $(\mathop{\rm{Dom}}(H^s))^*$, defined as
\begin{equation}\label{def_operator_Hs}
\df{(\mathop{\rm{Dom}}(H^s))^*}{H^s(w)}{v}{\mathop{\rm{Dom}}(H^s)}:=\int_{\R^{N+1}}(i\theta+|\xi|^2)^s\widehat{w}(\xi,\theta)\overline{\widehat{v}(\xi,\theta)} d \xi d\eta,
\end{equation}
for any $w, v \in \mathop{\rm{Dom}}(H^s)$.
We remark that $\mathop{\rm{Dom}}(H^s)$ coincides with the set of pointwise definition for the operator $H^{\frac{s}{2}}$.

The space $\mathop{\rm{Dom}}(H^s)$ can be characterized more explicitly. Let
\begin{equation}\label{deff_Xs}
X^s:=L^2(\R,W^{s,2}(\R^N)) \cap L^2(\R^N, W^{\frac{s}{2},2}(\R)),
\end{equation}
where we are identifying a  function $u:\R^{N+1}\to \R$ with variables $(x,t) \in \R^{N}\times \R$ with the map  $t \mapsto u(\cdot,t)$ or with the map  $x \mapsto u(x,\cdot)$ respectively.
We can endow $X^s$ with the natural norm
\begin{multline}\label{def_norm_Xs}
\norm{\phi}_{X^s}:=\norm{\phi}_{L^2(\R,W^{s,2}(\R^N))}+\norm{\phi}_{L^2(\R^N, W^{\frac{s}{2},2}(\R))}\\
=\left(\int_{\R^{N+1}}|\xi|^{2s}|\mc{F}_x(w)(\xi,t)|^2 d \xi dt\right)^\frac12+
\left(\int_{\R^{N+1}}|\theta|^s|\mc{F}_t(w)(x,\theta)|^2dx d\theta\right)^\frac12,
\end{multline}
where $\mc{F}_x$ and $\mc{F}_t$ denote the Fourier transforms in the variables $x$ and $t$ respectively.

\begin{proposition}\label{prop_DomHs_Xs}
For any $s \in (0,1)$,
\begin{equation}\label{eq_Xs_domHs}
X^s=\mathop{\rm{Dom}}(H^s)
\end{equation}
with equivalent norms. In particular the natural embedding
\begin{equation}\label{domHs_inclusion_L2Hs}
\mathop{\rm{Dom}}(H^s) \hookrightarrow L^2(\R,W^{s,2}(\R^N))
\end{equation}
is linear and continuous.
\end{proposition}
\begin{proof}
It is clear that  $|\xi|^{2s}\le |i\theta+|\xi|^2|^{s}$  and  $|\theta|^{s}\le |i\theta+|\xi|^2|^{s}$ for any $(\theta,\xi) \in \R^{N+1}$.
Furthermore it is easy to see  that there exists a constant $C>0$, depending on $s$, such that for any $a,b \in [0,+\infty)$
\begin{equation}\label{ineq_abs}
(a+b)^s \le C (a^s +b^s).
\end{equation}
Indeed the function $f(\tau):=\frac{(\tau +1)^s}{1+\tau^s}$  is bounded in $[0,+\infty)$.
Hence  $|i\theta+|\xi|^2|^{s}\le C (|\theta|^s +|\xi|^{2s})$ for any $(\theta,\xi) \in \R^{N+1}$. Then \eqref{eq_Xs_domHs} follows from the Plancherel identity.
\end{proof}

We also have an inversion formula for the operator $H^s$.
\begin{definition}
Let $q \in [1,2]$ and $g \in L^q(\R^N \times
(-\infty, T))$, for any $T \in \R$, be fixed.  We define the
operator
\begin{equation}\label{def_Js}
J_s g (x,t):=\frac{1}{\Gamma(s)(4\pi)^{\frac{N}{2}}}\int_0^\infty\int_{\R^N}  \frac{e^{-\frac{|z|^2}{4\tau}}}{\tau^{\frac{N}{2}+1-s}} g(x-z,t-\tau)\, dz \, d\tau.
\end{equation}
\end{definition}
By \cite{T_pointwise},  the operator $J_s$ has the following properties:
\begin{align}
&H^s(\phi)=g   \text{ if and only if }   J_s(g)=\phi, \quad \text { for any } \phi \in \mc{S}(\R^{N+1}),\label{Js_properties_inversion}\\
&J_s (g) \in L^q_{loc}(\R^{N+1}),\label{Js_properties_integrability} \\
&J_s= J_\alpha J_\beta\text{ for any } \alpha, \beta>0 \text { such that } \alpha+\beta=s.\label{Js_properties_compoisiton}\\
&\text{ If } g \text{ is non negative, then } J_s(g) =0 \text{ in
} (-\infty,0) \text { if and only if } g=0 \text{ in }
(-\infty,0). \label{Js_properties_intial_datum}
\end{align}
In view of  \eqref{Js_properties_inversion}, we say that   $J_s$ is the inverse of the operator $H^s$.

In the spirit of  \cite{T_pointwise} and  \cite{T_frac_heat} it makes sense to give the following definitions for weak supersolution, subsolutions and solutions of the problem
\begin{equation}\label{prob_general}
\begin{cases}
H^s(w)=g &\text{  in  }\R^N\times (0,+\infty),\\
w=0 &\text{  in  }\R^N\times (-\infty,0].
\end{cases}
\end{equation}

\begin{definition}\label{def_very_weak_solutions}
Let $q \in [1,2]$,  $g \in L^q(\R^N \times (-\infty, T))$ for any
$T \in \R$ and suppose that $g(t,x)=0$ a.e. in $(-\infty,0)\times
\R^{N}$. Then,  we say that $w$ is a very weak
supersolution (subsolution)  of \eqref{prob_general} if $w \in
L^q_{loc}(\R^{N+1})$ and
\begin{equation}\label{eq_very_weak_solution}
w\ge (\le) \, J_s(g) \text{  a.e. in  }\R^{N+1}.
\end{equation}
If $w$ is both a supersolution and a subsolution of
\eqref{prob_general}, then, we say that $w$ is a
solution of \eqref{prob_general}.
\end{definition}
The  definition above already encodes that condition $w=0$ in
$\R^N\times (-\infty,0]$ by \eqref{Js_properties_intial_datum}.

We can also define energy solutions to \eqref{prob_general} for a  suitable class of  data $g$.

\begin{definition}\label{def_weak_solutions}
Let  $g \in L^1_{loc}(\R^{N+1})$ be non-negative and suppose that $g(t,x)=0$ a.e. in $(-\infty,0)\times \R^{N}$.
Assume that
\begin{equation}
\phi \mapsto \int_{0}^{\infty}\left(\int_{\R^N}g \phi\, dx \right) dt
\end{equation}
belongs to $(\mathop{\rm{Dom}}(H^s))^*$.
Then we say that $w$ is   a energy or  weak  supersolution (subsolution) of \eqref{prob_general} if $w \in \mathop{\rm{Dom}}(H^s)$,
\begin{equation}\label{eq_weak_solution}
\sideset{_{(\mathop{\rm{Dom}}(H^s))^*}}{_{\mathop{\rm{Dom}}(H^s)}}{\mathop{\left\langle H^s(w), \phi\right\rangle}}
\ge(\le)\int_{0}^{\infty}\left(\int_{\R^N}g \phi\, dx \right) dt, \\
\end{equation}
for any non-negative $\phi \in \mc{S}(\R^{N+1})$ and  $w=0$  in
$\R^N\times (-\infty,0]$. If $w$ is a supersolution and a
subsolution of  the equation in \eqref{prob_general}, we say that
$w$ is a solution  of the equation \eqref{prob_general}.
\end{definition}

To deal with   \eqref{prob_frac_heat_hardy} we recall the
following  fractional Sobolev inequality (see  \cite[Theorem
6.5]{DPV})
\begin{equation}\label{ineq_Sob_frac}
\left(\int_{\R^N} |\phi|^\frac{2N}{N-2s} \, dx \right)^\frac{N-2s}{N} \le S_{N,s} \int_{\R^N}\int_{\R^N} \frac{|\phi(x)-\phi(y)|^2}{|x-y|^{N+2s}} \, dx dy
\quad \text { for any } \phi \in C^{\infty}_c(\R^N).
\end{equation}

Let $f \in L^2(\R^{N+1})$ and suppose that $f(t,x)=0$ a.e. in $(-\infty,0)\times \R^{N}$.  Let  $w \in \mathop{\rm{Dom}}(H^s)$ be such  that
\begin{equation}\label{hp_int_w}
\int_{0}^{+\infty}\left(\int_{\R^N}|w|^{\frac{2Np}{N+2s}}\, dx \right)^{\frac{N+2s}{N}} \, dt <+\infty \quad \text{ or } \quad  w \in L^{2p}(\R^{N+1}),
\end{equation}
$w=0$  in  $\R^N\times (-\infty,0]$, and
\begin{equation}\label{eq_frac_heat_hardy_weak}
\sideset{_{(\mathop{\rm{Dom}}(H^s))^*}}{_{\mathop{\rm{Dom}}(H^s)}}{\mathop{\left\langle H^s(w), \phi\right\rangle}}
\ge(\le)\int_{0}^{\infty}\left(\int_{\R^N}\left(\frac{\la}{|x|^{2s}} w\phi +w^p\phi +f \phi\right)\, dx \right) dt, \\
\end{equation}
for any non-negative $\phi \in \mc{S}(\R^{N+1})$. Then $w$ is   a
weak  supersolution (subsolution) of problem
\eqref{prob_frac_heat_hardy}. Indeed  in view of \eqref{ineq_Hardy_frac_Herbst},
\eqref{domHs_inclusion_L2Hs}, \eqref{ineq_Sob_frac},
\eqref{hp_int_w},  and the H\"older
inequality, the right hand side of
\eqref{eq_frac_heat_hardy_weak}, as a function of $\phi$, belongs
to  $({\rm{Dom}}(H^s))^*$.

\subsection{Properties of the operators $J_s$ and $H^s$}\label{subsec_prop_Js_Hs}
In this section we prove some preliminary results about the operators $J_s$ and $H^s$.
\begin{lemma}\label{lemma_Jsg_in L^2}
Suppose that  $ g \in L^1(\R^{N+1})\cap L^2(\R^{N+1})$. Then $J_s(g) \in L^2(\R^{N+1})$.
\end{lemma}
\begin{proof}
We proceed in the spirit of \cite{T_pointwise}. Let
\begin{equation}
Y(z,\tau):=\frac{1}{\Gamma(s)(4\pi)^{\frac{N}{2}}}
\frac{e^{-\frac{|z|^2}{4\tau}}}{\tau^{\frac{N}{2}+1-s}} \mbox{
   where  } \tau>0  \mbox{  and  }z\in \R^n.
\end{equation}
We can write $J_s(g)$ as
\begin{equation*}
J_s(g)= (Y\chi_{(0,1)}) \ast g +(Y\chi_{(1,+\infty)}) \ast g,
\end{equation*}
where for any measurable set $E\subseteq \R$
\begin{equation*}
\chi_E(t):=
\begin{cases}
1, &\text{ if } t  \in E,\\
0, &\text{ if } t  \not \in E.
\end{cases}
\end{equation*}
 By the proof of \cite[Theorem 2.1]{T_pointwise}, we know that  $Y\chi_{(0,1)} \in L^1(\R^{N+1})$ and so by the Young inequality $(Y\chi_{(0,1)}) \ast g  \in L^2(\R^{N+1})$,
since $g \in L^2(\R^{N+1})$. Furthermore, by the Young inequality
\begin{equation*}
    \norm{(Y\chi_{(1,+\infty)}) \ast h_n}_{L^2(\R^{N+1})} \le \norm{h_n}_{L^1(\R^{N+1})}\norm{Y\chi_{(1,+\infty)}}_{L^2(\R^{N+1})}
\end{equation*}
and by a change of variables
\begin{multline*}
    \norm{Y\chi_{(1,+\infty)}}_{L^2(\R^{N+1})} ^2= \frac{1}{(\Gamma(s))^2(4\pi)^{N}} \int_{1}^{+\infty} \int_{\R^N} \frac{e^{-\frac{|z|^2}{2\tau}}}{\tau^{N+2-2s}} \, dz d \tau \\
    =\frac{1}{(\Gamma(s))^2(4\pi)^{N}} \int_{1}^{+\infty} \tau^{-\frac{N}{2}-2+2s} \,d \tau \int_{\R^N}e^{-\frac{|z|^2}{2}} \, dz <+\infty,
\end{multline*}
since $N\ge2$ and $s \in (0,1)$. In conclusion $J_s(g) \in L^2(\R^{N+1})$.
\end{proof}

\begin{lemma}\label{lemma_Jsg_in DHs}
Let $ g \in L^2(\R^{N+1})$ and suppose that $J_s(g) \in L^2(\R^{N+1})$. Then
\begin{equation}\label{ineq_J_sg}
\int_{\R^{N+1}}|i\theta+|\xi|^2|^{2s}|\widehat{J_s(g)}(\xi,\theta)|^2d\xi \,d\theta <+\infty.
\end{equation}
In particular, $J_s(g) \in\mathop{\rm{Dom}}(H^s)$ and
$H^s(J_s(g))$ is well defined pointwise.
\end{lemma}
\begin{proof}
Since
\begin{equation*}
\mathcal{F}(J_{\frac{s}{2}} g)=\frac{1}{\Gamma(s)(4\pi)^{\frac{N}{2}}}\mathcal{F}\left( \chi_{[0,+\infty)}(\tau)\frac{e^{-\frac{|z|^2}{4\tau}}}{\tau^{\frac{N}{2}+1-s}}\right)\widehat{g},
\end{equation*}
we need to compute the Fourier transform of   $\chi_{(0,+\infty)}(\tau) \frac{e^{-\frac{|z|^2}{4\tau}}}{\tau^{\frac{N}{2}+1-s}}$, where
\begin{equation*}
\chi_{[0,+\infty)}(\tau)=
\begin{cases}
1, &\text{ if } \tau \in (0,+\infty),\\
0, &\text{ if } \tau \in (-\infty,0).
\end{cases}
\end{equation*}
We make this computation  in details for the sake of completeness, see  \cite{T_pointwise} and the references within. The Fourier transform of Gaussian function  and a change of variables yield
\begin{multline}
\int_{\R^{N+1}}e^{-i(z\cdot \xi+\tau \theta)}\chi_{[0,+\infty)}(\tau) \frac{e^{-\frac{|z|^2}{4\tau}}}{\tau^{\frac{N}{2}+1-s}} \, dz d\tau
=\int_0^\infty \tau^{-\frac{N}{2}-1+s}e^{-i\tau \theta} \int_{\R^N} e^{-i z \cdot \xi} e^{-\frac{|z|^2}{4\tau}}  \, dz d\tau\\
=2^N \pi^{\frac{N}{2}}\int_0^\infty \tau^{-1+s}e^{-\tau(i \theta+|\xi|^2)} \, d\tau = 2^N \pi^{\frac{N}{2}} \Gamma(s)(i \theta+|\xi|^2)^{-s}.
\end{multline}
Hence, since $g \in L^2(\R^{N+1})$, we have proved \eqref{ineq_J_sg} and the last claim follows from \eqref{ineq_J_sg}.
\end{proof}

\begin{lemma}\label{lemma_H_j_adjoint}
For any $\phi,\psi \in \mc{S}(\R^{N+1})$, we have that
\begin{align}
&\int_{\R^{N+1}}H^s(\phi) \psi \, dx dt=\int_{\R^{N+1}}\phi(-x,-t) H^s(\psi(-\cdot)) \, dx dt \label{eq_H_adjoint}\\
&=\int_{\R^{N+1}}H^{\frac{s}{2}}(\phi)(-x,-t) H^{\frac{s}{2}}(\psi(-\cdot)) \, dx dt, \notag\\
&\int_{\R^{N+1}}J_s(\phi) \psi \, dx dt=\int_{\R^{N+1}}\phi(-x,-t) J_s(\psi(-\cdot)) \, dx dt. \label{eq_J_adjoint}
\end{align}
\end{lemma}
\begin{proof}
By \eqref{def_Hs_fourier} and the Plancherel identity
\begin{equation}\label{}
\int_{\R^{N+1}}H^s(\phi) \psi \, dx dt=\int_{\R^{N+1}}(i\theta+|\xi|^2)^s\widehat{\phi}(\xi,\theta)\overline{\widehat{\psi}(\xi,\theta)} d \xi d\eta
\end{equation}
and so the first equality in \eqref{eq_H_adjoint} follows from $\overline{\widehat{\psi}}=\widehat{\psi(-\cdot)}$ and the Plancherel identity.
The second identity is an easy consequence of the first one. On the other hand by  \eqref{def_Js}, the Fubini-Tonelli theorem and a change of variables
\begin{multline}
\int_{\R^{N+1}}J_s(\phi) \psi \, dx dt
=\frac{1}{\Gamma(s)(4\pi)^{\frac{N}{2}}} \int_0^\infty\int_{\R^N}\int_{\R^{N+1}}\frac{e^{-\frac{|z|^2}{4\tau}}}{\tau^{\frac{N}{2}+1-s}} \phi(x-z,t-\tau)\psi(x,t)\, dx \, dt \, dz \, d\tau\\
=\frac{1}{\Gamma(s)(4\pi)^{\frac{N}{2}}}\int_{\R^{N+1}}\int_0^\infty\int_{\R^N}\frac{e^{-\frac{|z|^2}{4\tau}}}{\tau^{\frac{N}{2}+1-s}}\psi(x+z,t+\tau) \, dz \, d\tau \phi(x,t)\, dx \, dt,
\end{multline}
which yields \eqref{eq_J_adjoint}.
\end{proof}

\begin{lemma}\label{lemma_Hs2_Js2}
Suppose that $g \in L^2(\R^{N+1})$, $g(t,x)=0$ a.e. in $(-\infty,0)\times \R^{N}$ and that  $w \in \mathop{\rm{Dom}}(H^s)$  is a weak supersolution of the problem
\begin{equation}
\begin{cases}
H^s(w)= g, &\text{ in } \R^N\times (0,\infty),\\
w=0, &\text{ in } \R^N\times (-\infty,0),
\end{cases}
\end{equation}
in the sense given by Definition \ref{def_weak_solutions}. Then
\begin{equation}\label{ineq_w_s2}
H^{\frac{s}{2}}(w) \ge J_{\frac{s}{2}}(g)\quad \text{ a.e. in } \R^{N+1}.
\end{equation}
\end{lemma}

\begin{proof}
By \eqref{eq_J_adjoint} and a change of variables for any positive $\phi \in \mc{S}(\R^{N+1})$,
\begin{equation}
\int_{0}^{\infty}\left(\int_{\R^N}g \phi\, dx \right) dt
=\int_{\R^{N+1}}J_{\frac{s}{2}}(g)(-x,-t)H^{\frac{s}{2}}(\phi(-\cdot))\, dx\, dt.
\end{equation}
Hence by \eqref{eq_H_adjoint} we deduce that for any positive  $\phi \in \mc{S}(\R^{N+1})$
\begin{equation}\label{key}
\int_{\R^{N+1}}[H^{\frac{s}{2}}(w)(-x,-t)- J_{\frac{s}{2}}(g)(-x,-t)] H^{\frac{s}{2}}(\phi(-\cdot)) \, dx dt \ge 0.
\end{equation}
For any positive  $\psi \in \mc{S}(\R^{N+1})$  let
\begin{equation}
\phi:=J_{\frac{s}{2}}(\psi)(-\cdot).
\end{equation}
Then $H^{\frac{s}{2}}(\phi(-\cdot))=\psi$, and $\phi \in \mc{S}(\R^{N+1})$. It follows that
\begin{equation}
H^{\frac{s}{2}}(w)(-x,-t)- J_{\frac{s}{2}}(g)(-x,-t) \ge 0 \text{ a.e. in } \R^{N+1},
\end{equation}
which is equivalent to \eqref{ineq_w_s2}.
\end{proof}

\section{Ground state transformation for  the operator $H^s-\frac{\la}{|x|^{2s}}$.}\label{section_ground_state}
We start by recalling the following proposition.
\begin{proposition}\cite[Lemma 3.2]{FLS}, \cite[Proposition 3.2]{FF} \label{prop_mu}
Let $\Lambda_{N,s}$ be as in \eqref{def_La}. Then the function  $\Upsilon:\Big[0,\frac{N-2s}{2}\Big) \to (0,\Lambda_{N,s}]$ defined as
\begin{equation}\label{def_F}
\Upsilon(\alpha):=2^{2s}\frac{\Gamma\left(\frac{N+2s+2\alpha}{4}\right)\Gamma\left(\frac{N+2s-2\alpha}{4}\right)}{\Gamma\left(\frac{N-2s-2\alpha}{4}\right)\Gamma\left(\frac{N-2s+2\alpha}{4}\right)}
\end{equation}
is well-defined, continuous, surjective,  decreasing, $\Upsilon(0)=\Lambda_{N,s}$ and $\lim_{\alpha \to (\frac{N-2s}{2})^-}\Upsilon(\alpha)=0$.
\end{proposition}

Let us define  for any  $\la \in (0,\Lambda_{N,s}]$
\begin{equation}\label{def_mu}
\mu(\la):=\frac{N-2s}{2}-\Upsilon^{-1}(\la).
\end{equation}
Then clearly
\begin{equation}\label{ineq_mu_la}
0<{\mu(\la)}<\frac{N-2s}{2} \quad \text{ for any } \la \in (0,\Lambda_{N,s}).
\end{equation}

The most important result of this  section is a ground state
transformation for the operator $H^s-\frac{\la}{|x|^{2s}}$, in the
spirit of \cite[Proposition 4.1]{FLS} where a similar
representation was computed for the elliptic operator
$(-\Delta)^s-\frac{\la}{|x|^{2s}}$. We need a preliminary lemma
about radial functions.
\begin{lemma}\label{lemma_radial}
Let $f \in L^1(\R^N)$ be radial. Then there exists a constant $K_{N,s,\la}>0$ depending only on $N,s $ and $\la$ such that
\begin{equation}\label{ineq_radial}
\int_{\R^N}\frac{|f(y)|}{|x-y|^{\mu(\la)}} dy \le K_{N,s, \la}
|x|^{-\mu(\la)} \int_{\R^N}|f(y)| \, dy,\quad \text{ for any } x
\in \R^N \setminus\{0\}.
\end{equation}
\end{lemma}
\begin{proof}
Passing in polar coordinates and letting $y'=\frac{y}{\rho}$ and   $x'=\frac{y}{r}$  we obtain
\begin{multline}
\int_{\R^N}\frac{|f(y)|}{|x-y|^{\mu(\la)}} dy = \int_0^{+\infty}\rho^{N-1}|f(\rho)|\int_{\mb{S}^{N-1}}|x'r-\rho y'|^{-\mu(\la)}dS_{y'} \, d\rho \\
= r^{-\mu(\la)}\int_0^{+\infty}(r\sigma)^{N-1}|f(r\sigma)|\left(\int_{\mb{S}^{N-1}}|x'-\sigma y'|^{-\mu(\la)}dS_{y'}\right) r \,d\sigma,
\end{multline}
by the change of variables $\sigma=\frac{\rho}{r}$.
Let us define
\begin{equation}\label{def_K}
    K(\sigma):=\int_{\mb{S}^{N-1}}|x'-\sigma y'|^{-\mu(\la)}dS_{y'}.
\end{equation}
Since $|x'-\sigma y|^2=1-2\sigma x'\cdot y' +\sigma^2$,  with a
change of variables,  we deduce that 
\begin{equation}
    K(\sigma)=\int_{\mb{S}^{N-1}}|1-2\sigma  y_1+\sigma^2|^{-\frac{\mu(\la)}{2}}dS_{y'},
\end{equation}
and so  $K$ does not depends on $x$. Using spherical coordinates, see also   \cite{FV_radial}, we obtain
\begin{equation}
    K(\sigma)=C\int_0^\pi\frac{\sin(\theta)^{N-2}}{|1-2\sigma \cos(\theta)+\sigma^2|^{\frac{\mu(\la)}{2}}}d\theta,
\end{equation}
where $C>0$ is a positive constant depending only on $N$. Since for $\sigma=1$ the singularity in $0$ is integrable in view of \eqref{ineq_mu_la},
we conclude that $K$ is bounded.
It follows that, thanks to the change of variables $\rho= r \sigma $,
\begin{multline}
\int_{\R^N}\frac{|f(y)|}{|x-y|^{\mu(\la)}} dy   \le r^{-\mu(\la)}\norm{K}_{L^\infty(0,+\infty)} \int_0^{+\infty}(r\sigma)^{N-1}|f(r\sigma)|r \,d\sigma \\
=r^{-\mu(\la)}\norm{K}_{L^\infty(0,+\infty)} \int_0^{+\infty}\rho^{N-1}|f(\rho)| \,d\rho=r^{-\mu(\la)}\norm{K}_{L^\infty(0,+\infty)}\int_{\R^N}|f(y)| dy,
\end{multline}
and so we have proved \eqref{ineq_radial}.
\end{proof}

\begin{proposition}\label{prop_ground_state}
Let $ \la \in (0,\Lambda_{N,s})$ and let, for any $(x,t) \in
\R^{N}\setminus\{0\}\times \R$, and $\phi \in \mc{S}(\R^{N+1})$,
\begin{equation} \label{def_L}
L^s\phi(x,t):=\frac{1}{(4\pi)^{\frac{N}{2}}\Gamma(-s)}
\int_{-\infty}^t\int_{\R^N}\frac{\phi(x,t)-\phi(y,\sigma)}{|y|^{{\mu(\la)}}(t-\sigma)^{\frac{N}{2}+s+1}}e^{-\frac{|x-y|^2}{4(t-\sigma)}}dyd\sigma.
\end{equation}
Then there exists   $C>0$, depending only on $N,s$ and $\la$, such that for any $(x,t) \in \R^{N}\setminus\{0\}\times \R$
\begin{equation}\label{ineq_Ls}
|L^s\phi(t,x)|\le
C|x|^{-\mu(\la)}\left(\norm{\phi}_{L^\infty(\R^{N+1})}+\norm{|\nabla_{(x,t)}\phi|}_{L^\infty(\R^{N+1})}+\norm{|\nabla^2_{x}\phi|}_{L^\infty(\R^{N+1})}\right),
\end{equation}
for any $\phi \in \mc{S}(\R^{N+1})$, where $\nabla_x^2$ denotes the matrix of the second derivates of $\phi$ with respect to the spacial variable $x$.
Furthermore  for any $\phi \in \mc{S}(\R^{N+1})$
\begin{equation} \label{eq_ground_state}
H^s(\phi)-\frac{\la}{|x|^{2s}}\phi=L^s(|x|^{\mu(\la)}\phi),
\end{equation}
where ${\mu(\la)}$ is as in \eqref{def_mu}.
\end{proposition}
\begin{proof}
By a change of variables it is clear that
\begin{multline}
|L^s\phi(x,t)| = \left|\frac{1}{(4\pi)^{\frac{N}{2}}\Gamma(-s)}\int_0^{+\infty}\int_{\R^N}\frac{\phi(x,t)-\phi(x-z,t-\tau)}{|x-z|^{{\mu(\la)}}\tau^{\frac{N}{2}+s+1}}e^{-\frac{|z|^2}{4\tau}}dzd\tau\right|\\
\le\frac{1}{(4\pi)^{\frac{N}{2}}|\Gamma(-s)|}\left|\int_0^1\int_{\R^N}\frac{\phi(x,t)-\phi(x-z,t-\tau)}{|x-z|^{{\mu(\la)}}\tau^{\frac{N}{2}+s+1}}e^{-\frac{|z|^2}{4\tau}}dzd\tau\right|\\
+C_1\norm{\phi}_{L^\infty(\R^{N+1})}\int_1^{+\infty}\int_{\R^N}\frac{e^{-\frac{|z|^2}{4\tau}}}{|x-z|^{{\mu(\la)}}\tau^{\frac{N}{2}+s+1}}dzd\tau
\end{multline}
for some constant positive $C_1$ depending on  $N$, and $s$. In
view of Lemma \ref{lemma_radial},
\begin{equation}\label{ineq_far_0}
\int_1^{+\infty}\int_{\R^N}\frac{e^{-\frac{|z|^2}{4\tau}}}{|x-z|^{{\mu(\la)}}\tau^{\frac{N}{2}+s+1}}dzd\tau \le C_2 |x|^{-\mu(\la)}
\end{equation}
for some positive constant $C_2$ depending on $N, s$ and $\la$.
Furthermore, letting
\begin{equation}\label{def_B1}
B'_1:=\{x\in \R^N:|x|<1\},
\end{equation}
we have that
\begin{multline}
\int_0^1\int_{\R^N}\frac{\phi(x,t)-\phi(x-z,t-\tau)}{|x-z|^{{\mu(\la)}}\tau^{\frac{N}{2}+s+1}}e^{-\frac{|z|^2}{4\tau}}dzd\tau\\
=\int_0^1\int_{\R^N\setminus
B'_1}\frac{\phi(x,t)-\phi(x-z,t-\tau)}{|x-z|^{{\mu(\la)}}\tau^{\frac{N}{2}+s+1}}e^{-\frac{|z|^2}{4\tau}}dzd\tau
+\int_0^1\int_{B'_1}\frac{\phi(x,t)-\phi(x-z,t-\tau)}{|x-z|^{{\mu(\la)}}\tau^{\frac{N}{2}+s+1}}e^{-\frac{|z|^2}{4\tau}}dzd\tau.
\end{multline}
We will estimate the  last two integrals. We start by noticing that
\begin{multline}
\int_0^1\int_{\R^N\setminus B'_1}\frac{\phi(x,t)-\phi(x-z,t-\tau)}{|x-z|^{{\mu(\la)}}\tau^{\frac{N}{2}+s+1}}e^{-\frac{|z|^2}{4\tau}}dzd\tau\\
=\int_0^1\int_{\R^N\setminus B'_1}\frac{\phi(x,t)-\phi(x-z,t)}{|x-z|^{{\mu(\la)}}\tau^{\frac{N}{2}+s+1}}e^{-\frac{|z|^2}{4\tau}}dzd\tau
+\int_0^1\int_{\R^N\setminus B'_1}\frac{\phi(x-z,t)-\phi(x-z,t-\tau)}{|x-z|^{{\mu(\la)}}\tau^{\frac{N}{2}+s+1}}e^{-\frac{|z|^2}{4\tau}}dzd\tau.
\end{multline}
Then the  change of variables  $z=-z$ yields
\begin{multline}
2\int_0^1\int_{\R^N\setminus B'_1}\frac{\phi(x,t)-\phi(x-z,t)}{|x-z|^{{\mu(\la)}}\tau^{\frac{N}{2}+s+1}}e^{-\frac{|z|^2}{4\tau}}dzd\tau \\
=\int_0^1\int_{\R^N\setminus B'_1}\frac{\phi(x,t)-\phi(x-z,t)}{|x-z|^{{\mu(\la)}}\tau^{\frac{N}{2}+s+1}}e^{-\frac{|z|^2}{4\tau}}dzd\tau+
\int_0^1\int_{\R^N\setminus B'_1}\frac{\phi(x,t)-\phi(x+z,t)}{|x+z|^{{\mu(\la)}}\tau^{\frac{N}{2}+s+1}}e^{-\frac{|z|^2}{4\tau}}dzd\tau\\
\le\int_0^1\int_{\R^N\setminus B'_1}\frac{|2\phi(x,t)-\phi(x+z,t)-\phi(x-z,t)|}{|x-z|^{{\mu(\la)}}\tau^{\frac{N}{2}+s+1}}e^{-\frac{|z|^2}{4\tau}}dzd\tau\\
+\int_0^1\int_{\R^N\setminus B'_1}\frac{|\phi(x,t)-\phi(x+z,t)|}{\tau^{\frac{N}{2}+s+1}} e^{-\frac{|z|^2}{4\tau}}\left|\frac{1}{|z+x|^{\mu(\la)}}-\frac{1}{|x-z|^{\mu(\la)}}\right|dzd\tau\\
\le \int_0^1\int_{\R^N\setminus B'_1}\frac{|2\phi(x,t)-\phi(x+z,t)-\phi(x-z,t)|}{|x-z|^{{\mu(\la)}}\tau^{\frac{N}{2}+s+1}}e^{-\frac{|z|^2}{4\tau}}dzd\tau\\
+\int_0^1\int_{\R^N\setminus B'_1}|z|\frac{|\phi(x,t)-\phi(x+z,t)|}{\tau^{\frac{N}{2}+s+1}} e^{-\frac{|z|^2}{4\tau}}\left(\frac{1}{|z+x|^{\mu(\la)}}+\frac{1}{|x-z|^{\mu(\la)}}\right)dzd\tau.
\end{multline}
Since $\phi \in \mc{S}(\R^N)$ we have that
\begin{align*}
&|\phi(x,t)-\phi(x+z,t)| \le \norm{|\nabla_x \phi|}_{L^\infty(\R^{N+1})} |z|,\\
&|2\phi(x,t)-\phi(x+z,t)-\phi(x-z,t)| \le \norm{|\nabla^2_x \phi|}_{L^\infty(\R^{N+1})} |z|^2.
\end{align*}
Hence in view of Lemma \ref{lemma_radial} and a change of variables   we conclude that there exists a positive constant $C_3>0$, depending only on $N,s$, and $\la$ such that
\begin{multline}
\int_0^1\int_{\R^N\setminus B'_1}\frac{\phi(x-z,t)-\phi(x-z,t)}{|x-z|^{{\mu(\la)}}\tau^{\frac{N}{2}+s+1}}e^{-\frac{|z|^2}{4\tau}}dzd\tau \\
\le C_3|x|^{-\mu(\la)} (\norm{|\nabla_x
\phi|}_{L^\infty(\R^{N+1})}+\norm{|\nabla^2_x
\phi|}_{L^\infty(\R^{N+1})}).
\end{multline}
Moreover  we note  that
\begin{equation}
|\phi(x,t)-\phi(x,t-\tau)| \le \tau \norm{\pd{\phi}{t}}_{L^\infty(\R^{N+1})}
\end{equation}
and so by  Lemma \ref{lemma_radial} and a change of variables there exists a constant $C_4>0$, depending only on $N,s$, and $\la$, such that
\begin{equation}
\int_0^1\int_{\R^N\setminus B'_1}\frac{\phi(x,t)-\phi(x,t-\tau)}{|x-z|^{{\mu(\la)}}}dzd\tau
\le C_4\norm{\pd{\phi}{t}}_{L^\infty(\R^N\setminus B'_1)}|x|^{-\mu(\la)}.
\end{equation}
By \cite[Corollary 1.4]{ST} there exists a positive constant $C_5>0$, depending only on $N,s$, such that
\begin{equation}
\frac{e^{-\frac{|z|^2}{4\tau}}}{\tau^{\frac{N}{2}+s+1}} \le \frac{C_5}{|z|^{N + 2+2s}+ \tau^{\frac{N + 2+2s}{2}}} \quad \text{ for any } z \in \R^{N+1}, \tau \in (0,\infty).
\end{equation}
Then arguing, as in Lemma  \ref{lemma_radial}, we have that there
exist positive constants $C_6,C_7,C_8$ and $C_9$ depending only on
$N,s$, and $\la$, such that
\begin{multline}
\int_0^1\int_{B'_1}\frac{|\phi(x,t)-\phi(x-z,t-\tau)|}{|x-z|^{\mu(\la)}\tau^{\frac{N}{2}+s+1}}e^{-\frac{|z|^2}{4\tau}}dzd\tau \le C_6  \norm{|\nabla\phi|}_{L^\infty(\R^{N+1})} \int_0^1\int_{B'_1}\frac{(|z|+\sqrt{\tau})^{-N-2s}}{|x-z|^{\mu(\la)}}dzd\tau \\
\le C_7 \norm{|\nabla\phi|}_{L^\infty(\R^{N+1})} \int_0^1\int_{B'_1}\frac{|z|^{-N-2s}+\tau^{-\frac{N+2s}{2}}}{|x-z|^{\mu(\la)}}dzd\tau \\
\le C_8  \norm{|\nabla\phi|}_{L^\infty(\R^{N+1})} |x|^{-\mu(\la)}\int_0^1\int_{B'_1}(|z|+\sqrt{\tau})^{-N-2s} dzd\tau \le   C_9 |x|^{-\mu(\la)}.
\end{multline}
To see that the last integral is convergent, it is enough to consider  polar coordinates in $B'_1 \times (0,1)$ taking $\rho^2$ instead of $\rho$ as radial coordinate in the time coordinate.
Hence we have proved \eqref{ineq_Ls}.

For any $ \phi \in \mc{S}(\R^{N+1})$  by \eqref{def_Hs_pointwise} and \cite[Remark 2.2]{ST}
\begin{multline}
H^s(|x|^{-\mu(\la)}\phi)=\frac{1}{(4\pi)^{\frac{N}{2}}\Gamma(-s)}
\int_{-\infty}^t\int_{\R^N}\frac{|x|^{-{\mu(\la)}}\phi(x,t)-|y|^{-{\mu(\la)}}\phi(y,\sigma)}{(t-\sigma)^{\frac{N}{2}+s+1}}e^{-\frac{|x-y|^2}{4(t-\sigma)}}dyd\sigma\\
=\frac{1}{(4\pi)^{\frac{N}{2}}\Gamma(-s)}\left[\phi(x,t)\int_{-\infty}^t\int_{\R^N}\frac{|x|^{-{\mu(\la)}}-|y|^{-{\mu(\la)}}}{(t-\sigma)^{\frac{N}{2}+s+1}}e^{-\frac{|x-y|^2}{4(t-\sigma)}}dyd\sigma +
L^s(\phi)\right].
\end{multline}
Since in view of \cite[Lemma 4.1]{Fpower} and \cite[Corollary 1.4]{ST},
\begin{equation*}
\frac{1}{(4\pi)^{\frac{N}{2}}\Gamma(-s)}
\int_{-\infty}^t\int_{\R^N}\frac{|x|^{-{\mu(\la)}}-|y|^{-{\mu(\la)}}}{(t-\sigma)^{\frac{N}{2}+s+1}}e^{-\frac{|x-y|^2}{4(t-\sigma)}}dyd\sigma\\
=(-\Delta)^s|x|^{-\mu(\la)}=\frac{\la}{|x|^{2s+{\mu(\la)}}},
\end{equation*}
we conclude that \eqref{eq_ground_state} holds.
\end{proof}

As a direct application of the ground state representation, we get
the next  Kato-type inequality for the operator
$\bigg(H^s-\frac{\la}{|x|^{2s}}\bigg)$.
\begin{proposition}\label{prop_kato}
Assume that $\phi\in \mc{S}(\R^{N+1})$ is a non-negative function. Then, for any $m>1$, we have
\begin{equation}\label{ineq_kato}
L^s(\phi^m)(x,t)\le m\phi^{m-1}(x,t)L^s(\phi)(x,t).
\end{equation}
\end{proposition}
\begin{proof}
We start by showing the next algebraic inequality
\begin{equation}\label{ineq_ab}
a^m-b^m\le ma^{m-1}(a-b) \quad \text{ for any } a,b\ge 0.
\end{equation}
To this end, let $\gamma(x)=1-x^m-m(1-x)$ for any $x \in [0,+\infty)$.  Since  $\gamma'(x)=m(1-x^{m-1})$,  it is clear that $\gamma(x)\le \gamma(1)=0$ for any $x \in [0,+\infty)$. Taking
$x=b/a$ we obtain \eqref{ineq_ab}. Hence
\begin{multline}
L^s(\phi^m)(x,t)=
\frac{1}{(4\pi)^{\frac{N}{2}}\Gamma(-s)}\int_{-\infty}^t\int_{\R^N}\frac{\phi^m(x,t)-\phi^m(y,\sigma)}{|y|^{{\mu(\la)}}
(t-\sigma)^{\frac{N}{2}+s+1}}e^{-\frac{|x-y|^2}{4(t-\sigma)}}dyd\sigma\\
\le m\phi^{m-1}(x,t)\dyle\frac{1}{(4\pi)^{\frac{N}{2}}\Gamma(-s)}\int_{-\infty}^t\int_{\R^N}\frac{\phi(x,t)-\phi(y,\sigma)}{|y|^{{\mu(\la)}}
(t-\sigma)^{\frac{N}{2}+s+1}}e^{-\frac{|x-y|^2}{4(t-\sigma)}}dyd\sigma\\
=m\phi^{m-1}(x,t)L^s(\phi)(x,t),
\end{multline}
and so we have proved \eqref{ineq_kato}.
\end{proof}

\begin{remark}\label{remark_inversion_Ls}
It is an open and interesting problem to obtain an inversion
formula for the operator $L^s$. It is object of
current investigation by the authors. The main difficulty is the
loss of stability using the Fourier transform of convolutions as
in \cite{T_pointwise}, under the presence of singular kernels.

Hence, we will follow a different approach, based on  the
extension results proved in \cite{BCS}, to construct positive
solutions and supersolutions of problem
\eqref{prob_frac_heat_hardy}.
\end{remark}

\section{A comparison principle}\label{section_comparison}
In this section we are going to prove a comparison principle for weak solutions using the extension procedure developed in \cite{BCS}.
\subsection{An extension result}\label{subsec_extension}
We will use the extension procedure of  \cite{BCS} (see also \cite{BG,NS,ST}) to localize the problem.
We denote as $(z,t)=(x,y,t)$ the variable in   $\R^{N} \times(0,+\infty) \times \R$ and $\R^{N+1}_+:=\R^{N} \times (0,+\infty)$.
Furthermore we use the symbols $\nabla $ and $\dive$ to  denote the gradient, respectively the divergence, with respect to the space variable $z=(x,y)$.

For any $p \in [1,\infty)$ and any open set $E \subseteq \R^{N+1}_+$, let
\begin{equation*}
L^p(E,y^{1-2s}):=\left\{V:E\to \R \text{ measurable}: \int_{E} y^{1-2s}|V|^p \, dz<+\infty\right\}.
\end{equation*}
If $E$ is an  open Lipschitz set contained in $\R^{N+1}_+$,   $H^1(E,y^{1-2s})$ is defined as  the completion of $C_c^{\infty}(\overline{E})$ with  respect to the norm
\begin{equation*}
\norm{\phi}_{H^1(E,y^{1-2s})}:=\left(\int_{E} y^{1-2s}(\phi^2 +|\nabla \phi|^2)\, dz\right)^{\frac12}.
\end{equation*}
In view of   \cite[Theorem 11.11, Theorem 11.2, 11.12 Remarks(iii)]{OK}  and the extension theorems for weighted Sobolev
spaces with weights in the Muckenhoupt's $\mathcal{A}_2$ class proved in \cite{CSK},  the space $H^1(E,y^{1-2s})$ admits a concrete characterization   as
\begin{equation*}
H^1(E,y^{1-2s})=\left\{V \in W^{1,1}_{loc}(E):\int_{E} y^{1-2s} (V^2+|\nabla V|^2)\, dz< +\infty\right\}.
\end{equation*}
By  \cite{LM} there exists a linear and continuous trace operator
\begin{equation}\label{def_trace}
\Tr:H^1(\R_+^{N+1},y^{1-2s}) \to W^{2,s}(\R^{N}).
\end{equation}
For any  $w \in \mathop{\rm{Dom}}(H^s)$ and any $(x,y,t) \in \R^{N+1}_+ \times \R$,  let
\begin{equation}\label{def_W}
W(x,y,t):=\frac{y^{2s}}{4^{\frac{N}{2}+s} \pi^{\frac{N}{2}}\Gamma(s)}\int_{0}^\infty \int_{\R^N} \frac{e^{-\frac{|\xi|^2+y^2}{4\tau}}}{\tau^{\frac{N}{2}+1+s}} w(x-\xi,t-\tau) \, d\xi d\tau.
\end{equation}
The function $W$ is an extension of $w$ which possesses good
properties, as it has been proved in  \cite{BG} and \cite{BCS},
see also \cite[Theorem 1]{NS}, \cite[Theorem 1.7]{ST} and
\cite[Section 3, Section 4]{BG}. More precisely there holds the
following theorem.
\begin{theorem}\cite[Theorem 4.1, Remark 4.3]{BCS}, \cite[Corollary 3.2]{BG} \label{theorem_extension}
Let $w \in \mathop{\rm{Dom}}(H^s)$ and let $W$ be  as in \eqref{def_W}. Then $W \in L^2(\R,H^1(\R^{N+1}_+,y^{1-2s}))$ and  weakly solves
\begin{equation*}
\begin{cases}
y^{1-2s}W_t-\dive(y^{1-2s}\nabla W)=0, &\text{ in } \R_+^{N+1} \times \R,\\
\Tr(W(\cdot,t))=w(\cdot,t), &\text{ on }\R^N, \text{ for a.e. } t \in \R,\\
-\lim\limits_{y \to 0^+}y^{1-2s}\pd{W}{y}= \kappa_s H^s(w), &\text{on }\R^N\times \R,
\end{cases}
\end{equation*}
in the sense that  for a.e. $T>0$ and any  $\phi \in C^\infty_c (\overline{\R_+^{N+1}} \times[0,T])$
\begin{multline}\label{eq_frac_extended_without_eq}
\int_{\R}\left(\int_{\R_+^{N+1}}y^{1-2s} W\phi_t  \, dz \, \right)dt -\int_{\R_+^{N+1}}y^{1-2s} W(z,T)\phi(z,T)  \, dz +\int_{\R_+^{N+1}}y^{1-2s} W(z,0)\phi(z,0)  \, dz\\
=\int_{\R}\left(\int_{\R_+^{N+1}}y^{1-2s} \nabla W \cdot \nabla \phi \, dz \, \right)dt-\kappa_s\df{(\mathop{\rm{Dom}}(H^s))^*}{H^s(w)}{\phi(\cdot,0,\cdot)}{\mathop{\rm{Dom}}(H^s)}, \notag
\end{multline}
with $\kappa_s:=\frac{\Gamma(1-s)}{2^{2s-1} \Gamma(s)}$.
\end{theorem}

\subsection{A comparison principle}\label{subsection_comparison}
Let us begin by proving a weak maximum principle for the extended problem.
\begin{proposition}\label{prop_comparison_extended_0}
Let $T>0$ and suppose that $W \in L^2((0,T),H^1(\R^{N+1}_+,y^{1-2s}))$ is a weak positive subsolution of the problem
\begin{equation}\label{prob_extended_comparison_0}
\begin{cases}
y^{1-2s}W_t-\dive(y^{1-2s}\nabla W)= 0, &\text{ in } \R_+^{N+1} \times (0,T),\\
-\lim\limits_{y \to 0^+}y^{1-2s}\pd{W}{y}= 0, &\text{ on }\R^N\times (0,T),\\
W(z,0)=0, &\text{ on }\R_+^{N+1},
\end{cases}
\end{equation}
in the sense that for  any positive  $\phi \in C^\infty_c (\overline{\R_+^{N+1}} \times[0,T])$
\begin{multline}\label{eq_extended_comparison_0}
\int_{0}^{T}\left(\int_{\R_+^{N+1}}y^{1-2s}  \nabla  W \cdot \nabla \phi \, dz \, \right)dt-\int_{0}^{T}\left(\int_{\R_+^{N+1}}y^{1-2s}  W\phi_t \, dz \, \right)dt\\
+\int_{\R_+^{N+1}}y^{1-2s} W(z,T)\phi(z,T)  \, dz \le 0.
\end{multline}
Then $W\equiv 0$ a.e. in $\R^{N+1}_+\times (0,T)$.
\end{proposition}

\begin{proof}

By an approximated procedure, for example the Faedo-Galerkin method, we may suppose that $W_t \in  L^2((0,+\infty),L^2(\R^{N+1}_+,y^{1-2s}))$.
Hence , we may test \eqref{eq_extended_comparison_0} with $W$. Since
\begin{equation}
\int_{0}^{T}\left(\int_{\R_+^{N+1}}y^{1-2s}  W W_t \, dz \, \right)dt=\frac{1}{2}\int_{\R^{N+1}_+}y^{1-2s}  W^2(z,T)\, dz
\end{equation}
we obtain
\begin{equation}
\int_{0}^{T}\left(\int_{\R_+^{N+1}}y^{1-2s}  |\nabla W |^2 \, dz \, \right)dt +\frac{1}{2}\int_{\R^{N+1}_+}y^{1-2s}  W^2(z,T)\, dz \le 0,
\end{equation}
in view of \eqref{eq_extended_comparison_0}. We conclude that  $W\equiv 0$ in $\R^{N+1}_+$.
\end{proof}

\begin{corollary}\label{corollary_comparison_extended}
Let $T>0$ and suppose that   $V,W \in
L^2((0,T),H^1(\R^{N+1}_+,y^{1-2s}))$  satisfy the problem
\begin{equation}\label{prob_extended_comparison}
\begin{cases}
y^{1-2s}W_t-\dive(y^{1-2s}\nabla W)\le y^{1-2s}V_t-\dive(y^{1-2s}\nabla V), &\text{ in } \R_+^{N+1} \times (0,T),\\
-\lim\limits_{y \to 0^+}y^{1-2s}\pd{W}{y}\le -\lim\limits_{y \to 0^+}y^{1-2s}\pd{V}{y}, &\text{ on }\R^N\times (0,T),\\
 W(z,0) \le V(z,0) &\text{ a.e. } z \in \R_+^{N+1},
\end{cases}
\end{equation}
in a weak sense. Then $W\le V$ a.e. in $\R^{N+1}_+\times (0,T)$.
Furthermore, letting $v:=\Tr(V)$ and $w:=\Tr(W)$, we have that $w\le v $ a.e.  in $\R^{N} \times(0,+\infty)$.

\end{corollary}

\begin{proof}
Let $U:=W-V$. By the Kato inequality, and
\eqref{prob_extended_comparison}, it is easy to see that $U^+$
satisfies \eqref{prob_extended_comparison_0}. Hence $U^+\equiv 0$
in $\R^{N+1}_+\times (0,T)$ thanks to Proposition
\ref{prop_comparison_extended_0}, that is $W\le V$ a.e. in
$\R^{N+1}_+\times (0,T)$. Let us define
\begin{equation}\label{def_tilde_U}
\widetilde U(z,t)=\widetilde U(x,y,t):=
\begin{cases}
U(x,y,t), & \text { if }y>0,\\
U(x,-y,t), & \text { if }y<0,
\end{cases}
\end{equation}
and let $\{\rho_\e\}_{\e >0}$ be a family of standard mollifiers on
$\R^{N+1}$. By \cite[Lemma 1.5]{K_weighted_Sobolev}, $\rho_\e \ast
\widetilde U \to \widetilde U$ in $H^1(\R^{N+1},|y|^{1-2s})$ and
$\rho_\e \ast \widetilde{U} \ge 0 $  in  $\R^{N+1}\times
(0,+\infty)$. Since $\rho_\e \ast \widetilde U$ is smooth in
particular $\Tr(\rho_\e \ast  \widetilde U)\ge 0$ in
$\R^{N}\times (0,+\infty)$. Up to a subsequence, we may pass to
the limit as $\e \to 0^+$ and deduce that $\Tr(U) \ge 0$ a.e. in
$\R^{N} \times(0,+\infty)$ thanks to the continuity of the trace
operator defined in \eqref{def_trace}. We conclude that  $w\le v $
in $\R^{N} \times(0,+\infty)$.

\end{proof}

From Corollary   \ref{corollary_comparison_extended} and Theorem \ref{theorem_extension} we can deduce the following  weak comparison principle for $H^s$.
\begin{corollary}\label{corollary_comparison_Hs}
Suppose that $w,v \in \mathop{\rm{Dom}}(H^s)$ solve
\begin{equation}\label{prob_comparison_Hs}
\begin{cases}
H^s(w)\le H^s(v), & \text { in } \R^N \times (0,+\infty),\\
w=v=0, & \text { on } \R^N \times (-\infty,0],
\end{cases}
\end{equation}
in the sense that  $w= v=0$ on $\R^N \times (-\infty,0]$ and
\begin{equation}\label{eq_comparison_Hs}
\sideset{_{(\mathop{\rm{Dom}}(H^s))^*}}{_{\mathop{\rm{Dom}}(H^s)}}{\mathop{\left\langle H^s(w), \phi\right\rangle}}
\le \sideset{_{(\mathop{\rm{Dom}}(H^s))^*}}{_{\mathop{\rm{Dom}}(H^s)}}{\mathop{\left\langle H^s(v), \phi\right\rangle}}
\end{equation}
for any positive $\phi \in \mc{S}(\R^{N+1})$. Then $w\le v $  a.e. in $\R^{N+1}$.
\end{corollary}

Comparing a weak supersolution of problem \eqref{prob_frac_heat_hardy} with $0$, from Corollary \ref{corollary_comparison_Hs} we have the following.
\begin{corollary}\label{corollary_positivity}
Suppose that $f \in L^2(\R^{N+1})$, $f(t,x)=0$ a.e. in
$(-\infty,0)\times \R^{N}$ and that $f$  is non-negative. Then any
weak supersolution $w$ of problem \eqref{prob_frac_heat_hardy} in
the sense given by Definition \ref{def_weak_solutions}, satisfies $w\ge 0$ a.e. in $\R^{N+1}$.
\end{corollary}

\subsection{An existence result}\label{subsection_existence}
In this subsection we show how to deduce the existence of  solutions of problem \eqref{prob_frac_heat_hardy} from the existence of  supersolutions  by the means of the comparison principle proved in
Subsection \ref{subsection_comparison}.

Suppose that $f \in L^2(\R^{N+1})$, $f(t,x)=0$ a.e. in $(-\infty,0)\times \R^{N}$.
Let us define for any $n \in \mathbb{N}$ a cut-off function $\eta_n\in C^{\infty}_c(\R^{N+1})$ such that $\eta_n=1$ in $B'_{n} \times (\frac{1}{n+1},n+1)$ and $\eta_n=0$ in
$(\R^{N}\setminus B'_{n+2}) \times \R$ and in $\R^{N} \times  [(-\infty,\frac{1}{n+2})\cup (n+2,+\infty)]$.
Let us consider the  approximating problems
\begin{equation}\label{prob_approximated_0}
\begin{cases}
H^s(w_0)=\eta_0 \frac{f}{1+f}, &\text{ in } \R^{N}\times(0,+\infty),\\
w_0=0, &\text{ in } \R^{N}\times (-\infty,0],
\end{cases}
\end{equation}
and
\begin{equation}\label{prob_approximated_n}
\begin{cases}
H^s(w_{n+1})=\eta_n\left(\la \dfrac{w_n}{1+\frac{1}{n}w_n} \dfrac{1}{(|x|+\frac{1}{n})^{2s}} +  \dfrac{w^p_n}{1+\frac{1}{n} w^p_n} +\dfrac{f}{1+\frac{1}{n}f}\right), &\text{ in } \R^{N}\times(0,+\infty),\\
w_{n+1}=0, &\text{ in } \R^{N}\times (-\infty,0],
\end{cases}
\end{equation}
for any $n \in \mathbb{N}\setminus\{0\}$.
To simplify the notations we define
\begin{align}
&h_0:=\eta_0 f_0,
&h_n:=\eta_n\left(\la \frac{w_n}{1+\frac{1}{n}w_n} \frac{1}{(|x|+\frac{1}{n})^{2s}} +  \frac{w^p_n}{1+\frac{1}{n} w^p_n} +\frac{f}{1+\frac{1}{n}f}\right).
\end{align}

\begin{lemma}\label{lemma_existence_approximated_problems}
For any $n\in\mathbb{N}$, the function
\begin{equation}\label{def_approximated_sol}
w_n:=J_s\left(h_n\right),
\end{equation}
belongs to $\mathop{\rm{Dom}}(H^s)$, and it  is pointwise solutions to the  problem \eqref{prob_approximated_0} if $n=0$ or \eqref{prob_approximated_n} if $n>0$. Furthermore    $w_n$ solves problem \eqref{prob_approximated_0} if $n=0$ or \eqref{prob_approximated_n} if $n>0$ in the sense given by Definition \ref{def_weak_solutions}   and
\begin{equation}\label{ineq_wn}
0\le w_n\le w_{n+1} \text{ a.e. in } \R^{N+1} \quad \text{ for any } n \in \mathbb{N}.
\end{equation}
\end{lemma}

\begin{proof}
By Lemma \ref{lemma_Jsg_in L^2} and  by  Lemma \ref{lemma_Jsg_in
DHs}, it follows that $w_n \in \mathop{\rm{Dom}}(H^s)$ for any $n
\in \mathbb{N}$ and that $H^s(w_n)=h_n$ pointwise. Then for any
$\phi \in \mc{S}(\R^{N+1})$, by the Plancherel identity,
\begin{equation}
\int_{\R^{N+1}}(i\theta+|\xi|^2)^s \widehat{w}_n \overline{\widehat{\phi}} \, d\xi \,  d \theta = \int_{\R^{N+1}}H^s(w_n)\phi \, d\xi \, d \theta =\int_{0}^\infty\int_{\R^{N}} h_n \phi \, dx \, dt.
\end{equation}
Thanks to  the H\"older inequality, $\phi \mapsto \int_{\R^{N+1}}
h_n \phi \, dx \, dt$ belongs to  $(\mathop{\rm{Dom}}(H^s))^*$
while $w_n=0$ in $\R^{N} \times (-\infty,0]$ in view of
\eqref{Js_properties_intial_datum}.

In conclusion, $w_n$ is a solution of problem
\eqref{prob_approximated_0} if $n=0$ or
\eqref{prob_approximated_n} if $n>0$, in the sense given by
Definition \ref{def_weak_solutions}. Finally \eqref{ineq_wn}
follows from Corollary \ref{corollary_comparison_Hs} or directly
by \eqref{def_Js} and \eqref{def_approximated_sol}.
\end{proof}

In the next proposition we prove the existence of a weak solution of problem \eqref{prob_frac_heat_hardy} starting from a weak supersolution of the same problem.
\begin{proposition}\label{prop_existence_weak_solution}
Suppose that $f \in L^2(\R^{N+1})$, $f(t,x)=0$ a.e. in $(-\infty,0)\times \R^{N}$ and that  $u \in \mathop{\rm{Dom}}(H^s)$  is a weak supersolution of problem \eqref{prob_frac_heat_hardy} in the sense
given by Definition \ref{def_weak_solutions}. Let ${w_n}$ be as  in Lemma \ref{lemma_existence_approximated_problems}. Then  there exists a function $w \in \mathop{\rm{Dom}}(H^s)$ such that  $w_n(x,t)
\to w(x,t)$ as $n \to \infty$ a.e. in  $\R^{N+1}$ and $w$ weakly solves problem \eqref{prob_frac_heat_hardy}.
\end{proposition}
\begin{proof}
In view of Corollary \ref{corollary_comparison_Hs} and \eqref{ineq_wn} we conclude that $w:=\lim_{n \to \infty} w_n$ is well-defined, that $ w \le u$ a.e. in $\R^{N+1}$
and in particular that  $w \in L^2(\R^{N+1})$.
Let us prove that $\{w_n\}_{n \in \mathbb{N}}$ is bounded in $\mathop{\rm{Dom}}(H^s)$. To this end notice that $u$ is a weak supersolution of the problem
\begin{equation}
\begin{cases}
H^s(u) \ge h_n, &\text{ in } \R^{N+1} \times (0,+\infty),\\
u=0 &\text{ in }\R^{N+1} \times (-\infty,0),
\end{cases}
\end{equation}
since $w_n \le w \le u$ a.e. in $\R^{N+1}$. Then by the Plancherel identity and Lemma  \ref{lemma_Hs2_Js2}
\begin{multline}
\int_{\R^{N+1}}|(i\theta+|\xi|^2)^\frac{s}{2} \widehat{w}_n|^2 \, d\xi \,  d \theta=
\int_{\R^{N+1}}|\mathcal{F}^{-1}((i\theta+|\xi|^2)^\frac{s}{2} \widehat{w}_n)|^2 \, d\xi \,  d\theta=\int_{\R^{N+1}}|J_{\frac{s}{2}}(h_n)|^2 \, d\xi \,  d \theta\\
\le \int_{\R^{N+1}}|H^{\frac{s}{2}}(u)|^2 \, d\xi \,  d \theta=\int_{\R^{N+1}}|(i\theta+|\xi|^2)^\frac{s}{2} \widehat{u}|^2 \, d\xi \,  d \theta<+\infty.
\end{multline}
We conclude that, up to a subsequence, $w_n \to  w$ weakly in $\mathop{\rm{Dom}}(H^s)$ as $n \to \infty$. Thus, thanks to the Dominated Convergence Theorem,
\begin{multline}
\int_{\R^{N+1}}(i\theta+|\xi|^2)^s \widehat{w} \overline{\widehat{\phi}} \, d\xi \,  d \theta=
\lim_{n \to \infty}\int_{\R^{N+1}}(i\theta+|\xi|^2)^s \widehat{w}_n \overline{\widehat{\phi}} \, d\xi \,  d \theta\\
=\lim_{n \to \infty}\int_{0}^\infty\int_{\R^{N}} h_n \phi \, dx \, dt= \int_{0}^\infty\int_{\R^{N}} \left( \frac{\la}{|x|^{2s}}w+w^p+f\right) \phi \, dx \, dt,
\end{multline}
for any $\phi \in \mc{S}(\R^{N+1})$.

It is also clear that $\phi \mapsto
\int_{\R^{N+1}}(\frac{\la}{|x|^{2s}}w+w^p+f) \phi \, dx \, dt$
belongs to  $(\mathop{\rm{Dom}}(H^s))^*$ since  $w \le u$ a.e. in
$\R^{N+1}$. In conclusion  $w$ is a weak solution of  problem
\eqref{prob_frac_heat_hardy} in the sense given by Definition
\ref{def_weak_solutions}.
\end{proof}

We can also prove the existence of a very weak solution of problem \eqref{prob_frac_heat_hardy} starting from a very weak supersolution of the same problem.
\begin{proposition}\label{prop_existence_very_weak_solution}
Let $q \in [1,2]$,  $f\in L^q(\R^N \times (-\infty, T))$ for any $T \in \R$ and suppose that $f(x,t)=0$ a.e. in $(-\infty,0)\times \R^{N}$. Let $u $ be  a weak supersolution of problem
\eqref{prob_frac_heat_hardy} in the sense given by Definition \ref{def_very_weak_solutions}. Let ${w_n}$ be as  in Lemma \ref{lemma_existence_approximated_problems}. Then  there exists a function $w
\in L^q_{loc} (\R^{N+1})$ such that  $w_n(x,t)      \to w(x,t)$ as $n \to \infty$ a.e. in  $\R^{N+1}$ and $w$ is a very weak solution of problem \eqref{prob_frac_heat_hardy}.
\end{proposition}
\begin{proof}
By \eqref{def_Js} and \eqref{def_approximated_sol}, it is clear
that $w_n \le u$ a.e. in $\R^{N+1}$ for any $n \in \mathbb{N}$. Hence $w:=\lim_{n \to \infty} u_n$ is well-defined and  $ w \le u$
a.e. in $\R^{N+1}$, in particular $w \in L^q_{loc}(\R^{N+1})$.
Furthermore by the Dominated Converge Theorem
\begin{equation}
w =\lim_{n \to \infty} w_n =\lim_{n \to \infty} J_s(h_n) = J_s\left(\frac{\la}{|x|^{2s}}w+w^p+f\right),
\end{equation}
that is, $w$ is a very weak solution of \eqref{prob_frac_heat_hardy}.
\end{proof}

\section{Local behaviour of weak solutions}\label{sec_local_weak_sol}
In this section we will study the local behaviour near the singular point $0$ of non-negative weak supersolutions of the equation
\begin{equation}\label{eq_no_f_nor_datum}
H^s(w)=\la\frac{w}{|x|^{2s}} \quad \text{ in } \R^N\times (0,\infty).
\end{equation}
The initial datum plays no role in the results proved in this section, hence we consider the following definition of weak solutions.

\begin{definition}\label{def_weak_solutions_no_initail_datum}
We say that $w$ is   weak  supersolution (subsolution) of \eqref{eq_no_f_nor_datum} if $w \in \mathop{\rm{Dom}}(H^s)$,
\begin{equation}\label{eq_super_no_f_nor_datum_weak}
\sideset{_{(\mathop{\rm{Dom}}(H^s))^*}}{_{\mathop{\rm{Dom}}(H^s)}}{\mathop{\left\langle H^s(w), \phi\right\rangle}}
\ge(\le)\int_{0}^{\infty}\left(\int_{\R^N}  \frac{\la}{|x|^{2s}}w\phi\, dx \right) dt, \\
\end{equation}
for any non-negative $\phi \in \mc{S}(\R^{N+1})$ such that  $\phi=0$  in  $\R^N\times (-\infty,0]$.
If $w$ is a supersolution and a subsolution of  the equation in \eqref{eq_super_no_f_nor_datum_weak} we say that $w$ is a solution  of  \eqref{eq_super_no_f_nor_datum_weak}.
\end{definition}

We start by fixing some notations and recalling some useful preliminary results.
\begin{proposition}{\cite[Section 2]{FF}}\label{prop_ineq_hardy}
For any $\phi \in C^\infty_c(\overline{\R^{N+1}_+})$
\begin{equation}\label{ineq_Hardy_frac}
k_s \Lambda_{N,s} \int_{\R^N}\frac{|\phi(x,0)|^2}{|x|^{2s}} \, dx \le \int_{\R^{N+1}_+} y^{1-2s}|\nabla \phi|^2\, dz,
\end{equation}
with $\Lambda_{N,s}$ as in \eqref{def_La} and    $k_s$ as in Theorem \ref{theorem_extension}.
\end{proposition}
For $r>0$ let us define
\begin{equation}\label{def_Br+_Br_Br'}
    B_r:=\{z \in \R^{N+1}:|z|<r\}, \quad  B_r^+:=\{z \in \R_+^{N+1}:|z|<r\}, \quad  \text{ and } \quad B_r':=\{x \in \R^{N}:|x|<r\}.
\end{equation}

\begin{proposition}\cite[Lemma 4.1]{Fpower} \label{prop_Phi}
For any $\la \in (0,\Lambda_{N,s})$ there exists  a positive, continuous function $\Phi_\la:\R^{N+1}_+\setminus\{0\} \to \R$ such that  $\Phi_\la \in H^1(B_r^+,y^{1-2s})$ for any $r>0$ and $\Phi_\la$ weakly solves the problem
\begin{equation}\label{prob_Phi}
\begin{cases}
\dive(y^{1-2s} \nabla \Phi_\la)=0, &\text{ in } \R^{N+1}_+,\\
\Phi_\la(x,0)=|x|^{-\mu(\la)}, &\text{ on } \R^N, \\
-\lim_{y \to 0^+}y^{1-2s}\pd{\Phi_\la}{y}= \kappa_s \frac{\la}{|x|^{2s}} |x|^{-\mu(\la)}, &\text{ on } \R^N,
\end{cases}
\end{equation}
in the sense that for any $\varphi \in C_c^{\infty}(\overline{\R^{N+1}_+})$
\begin{equation}\label{eq_Phi}
\int_{\R^{N+1}_+}y^{1-2s} \nabla \Phi_\la \nabla \varphi\, dz =\kappa_s \int_{\R^N}\frac{\la}{|x|^{2s}} |x|^{-\mu(\la)} \varphi\, dz.
\end{equation}
Furthermore
\begin{align}
&\Phi_\la(\tau z)=\tau^{-\mu(\la)}\Phi_\la(z), \quad \text{ for any } \tau \in [0,\infty)  \text{ and any } z \in \R^{N+1}_+\setminus\{0\},\label{Phi_omogeneity}   \\
&C_1 |z|^{-\mu(\la)}\le \Phi_\la(z)\le  C_2 |z|^{-\mu(\la)}, \quad \text{ for any } \tau \in (0,\infty) \text{ and any }  z \in \R^{N+1}_+\setminus\{0\},\label{ineq_Phi_|z|}\\
&\nabla \Phi_\la(z)\cdot z = -\mu(\la) \Phi_\la(z), \quad \text{ for any }  z \in\R^{N+1}_+\setminus\{0\}, \label{eq_Phi_nabla}
\end{align}
for some positive constants $C_1,C_2>0$.
\end{proposition}

We recall the definition of the class $\mathcal{A}_2$ of Muckenhoupt weights.
\begin{definition} We say that  a measurable function $\rho:\R^{N+1} \to [0,+\infty)$ belongs to  $\mathcal{A}_2$ if
\begin{equation}\label{def_A2}
\sup_{r>0}\left\{r^{-2N-2}\left(\int_{B_r} \rho \, dz\right)\left(\int_{B_r} \rho^{-1} \, dz\right)\right\} < +\infty.
\end{equation}
\end{definition}

Let
\begin{equation}\label{def_tilde_Phi}
\widetilde\Phi_\la(z)=\widetilde\Phi_\la(x,y):=
\begin{cases}
\Phi_\la(x,y), & \text { if }y>0,\\
\Phi_\la(x,-y), & \text { if }y<0.
\end{cases}
\end{equation}

\begin{proposition}\label{prop_Phi_Muckenhoupt}
We have that $|y|^{1-2s}\widetilde\Phi_\la^2 \in \mathcal{A}_2$, that is  $|y|^{1-2s}\widetilde\Phi_\la^2$ is a Muckenhoupt weight.
\end{proposition}
\begin{proof}
For any $r>0$,  passing in polar coordinates,
\begin{multline}
\int_{B_r}|y|^{1-2s}\widetilde\Phi_\la^2 \, dz \int_{B_r}|y|^{2s-1}\widetilde\Phi_\la^{-2} \, dz  \\
\le \con \int_0^r \rho^{N+1-2s-2{{\mu(\la)}}} \, d\rho  \int_0^r \rho^{N+2s-1+2{{\mu(\la)}}} \, d\rho
= \con r^{2N+2},
\end{multline}
in view of \eqref{def_mu}, \eqref{ineq_Phi_|z|} and  \eqref{def_tilde_Phi}. Hence the claim follows from \eqref{def_A2}.
\end{proof}

For any $r>0$ let us define the Hilbert spaces
\begin{align}
&L^2(B_r^+,y^{1-2s}\Phi_\la^2):=\{U \in L^1_{loc}(B_r^+):\int_{B_r^+} y^{1-2s}\Phi_\la^2 |U|^2 \, dz <+\infty\}, \label{def_L2_Phi_weighted}\\
&H^1(B_r^+,y^{1-2s}\Phi_\la^2):=\{U \in W^{1,1}_{loc}(B_r^+):\int_{B_r^+} y^{1-2s}\Phi_\la^2(|U|^2+|\nabla U|^2 )\, dz <+\infty\}, \label{def_H1_Phi_weighted}
\end{align}
endowed with the natural norms
\begin{align}\label{def_norms_Phi_weighted}
&\norm{U}_{L^2(B_r^+,y^{1-2s}\Phi_\la^2)}:=\int_{B_r^+} y^{1-2s}\Phi_\la^2 |U|^2 \, dz, \\
&\norm{U}_{H^1(B_r^+,y^{1-2s}\Phi_\la^2)}:=\int_{B_r^+} y^{1-2s}\Phi_\la^2(|U|^2+|\nabla U|^2 ) \, dz.
\end{align}

In the next proposition we estimate the behaviour of weak solutions of  \eqref{eq_no_f_nor_datum} near the origin.
\begin{proposition}\label{prop_estimate_singularity}
Assume that $w$ is a non-negative, non-trivial, weak  supersolution of
\begin{equation}\label{eq_supersolution_singularity}
H^s(w)= \frac{\la}{|x|^{2s}}w \quad \text{ in } \R^N \times (0,+\infty).
\end{equation}
Then for $r>0$ small and for all $(t_1,t_2)\subset
\subset(0,\infty)$,  there exists a positive
constant $C(r,t_1,t_2,w)$ depending only on $r,t_1,t_2$ and $w$, such that
\begin{equation}\label{ineq_singularity}
w(x,t)\ge C(r,t_1,t_2,w) |x|^{-\mu(\la)} \quad \text{ for a.e. }
(x,t)\in B_r(0)\times (t_1,t_2).
\end{equation}
In particular
\begin{equation}
w>0 \quad \text{ a.e. in } \R^{N}\times(0,+\infty).
\end{equation}
\end{proposition}
\begin{proof}
Let $W$ be the extension of $w$ as in Theorem \ref{theorem_extension}. Then,  by \eqref{def_W}, $W>0$  in $\R^{N+1}_+$ while by classical parabolic regularity theory $W$ is continuous in $\R^{N+1}_+$. Let $r>0$ and $(t_1,t_2)\subset \subset(0,\infty)$. For any $\rho>0$ there exists $\delta>0$ such that  $W>\rho$ on $S_{4r,\delta}\times(t_1,t_2)$ where for any $r>0$ and any $\delta \in [0,4r)$
\begin{equation}\label{def_S_rdelta}
S_{4r,\delta}:=\{z=(x,y) \in \R^{N+1}_+:|z|=4r, y > \delta\}.
\end{equation}
Let $\eta\in C^{\infty}(\R^{N+1})$ be a cut off-function such that  $\eta(x,y)=0$ if $y <\delta$ and $\eta(x,y)=1$ if $y >\frac{4r+\delta}{2}$ for any $x \in \R^{N}$.   Let us consider the elliptic problem
\begin{equation}\label{poof:estimate_singularity:1}
\begin{cases}
-\dive(y^{1-2s}\nabla U)=0, &\text{ in } B^+_{4r},\\
U=\rho\eta &\text{ on } S^+_r, \\
-\lim\limits_{y \to 0^+}y^{1-2s}\pd{U}{y}= \kappa_s  \frac{\la}{|x|^{2s}}u, &\text{ on }B_{4r}',
 \end{cases}
\end{equation}
where $u=\Tr(U)$.
In view of \eqref{ineq_Hardy_frac}, by standard minimization technique it is easy to see that there exists  a  unique weak solution $u \in H^1(B_r^+,y^{1-2s})$, that is
$u$ solves the equation
\begin{equation*}
\int_{B_{4r}^+}y^{1-2s}\nabla U \cdot \nabla \phi  \, dz=\kappa_s \la\int_{B'_{4r}}y^{1-2s}\frac{u}{|x|^{2s}} \phi(x,0) \, dx
\end{equation*}
for any $\phi \in C^{\infty}(\overline{B_{4r}^+})$ such that $\phi=0$ on $S_{4r}^+$.
Furthermore since $\rho \eta$ is positive, it is clear that  $U^-=0$ on  $S_{4r}^+$. Hence we may test with $U^-$ the equation above and conclude that $U$ is non-negative.
Since $U=\rho \eta$ on $S^+_r$, it is also clear that $U\not \equiv 0$.

Let us define $V:=U \Phi_\la^{-1}$ and $v:=u|\cdot|^{-\mu(\la)}$. Then with a direct computation we see  that $V \in H^1(B_{4r}^+,y^{1-2s}\Phi_\la^2)$ and  weakly solves the problem
\begin{equation*}
\begin{cases}
-\dive(y^{1-2s}\Phi_\la^2\nabla V)=0, &\text{ in } B^+_{4r},\\
V=\rho \eta\Phi_\la^{-1} &\text{ on } S^+_{4r}, \\
-\lim\limits_{y \to 0^+}y^{1-2s}\pd{V}{y}= 0, &\text{ on }B_{4r}'.
\end{cases}
\end{equation*}
Let us define $\widetilde{V}$ as
\begin{equation}\label{def_tildeV}
\widetilde{V}(x,y):=
\begin{cases}
V(x,y), &\text{ if } y>0,\\
V(x,-y), &\text{ if } y<0.
\end{cases}
\end{equation}
Then $\widetilde{V}\in H^1(B_{4r},|y|^{1-2s}\Phi_\la^2)$ is a weak solution of the equation
\begin{equation}\label{eq_supersolution}
-\dive(|y|^{1-2s}\Phi_\la^2\nabla \widetilde{V})=0 \quad \text{ in } B_{2r}.
\end{equation}
Then  there exists a positive constant $C>0$ such that $V\ge C$
a.e. in $B_{2r}$  thanks to Proposition \ref{prop_Phi_Muckenhoupt}
and the Harnack inequality proved in \cite{CS_harnack_weighted}.
It follows that $U\ge C \Phi_\la$  a.e. in  $B_{2r}$.

Let us define for any $t\in \R$  the function $\overline{U}(t,z):=U(z)$ for any $z \in B_{4r}^+$.
Since $W\ge \overline{U}$ on $S_{4r}^+ \times (t_1,t_2)$, arguing as in Subsection \ref{subsection_comparison},  we can show that $W \ge \overline{U}$ a.e. in $B_{4r}^+$. It follows that there exists
a positive constant $C(r,t_1,t_2,w)>0$ depending only on $r,t_1,t_2$ and $w$ such that
\begin{equation*}
W \ge C(r,t_1,t_2,w) \Phi_\la \quad \text{ a.e. in } B_{2r}^+.
\end{equation*}
Let us  define
\begin{equation*}\label{def_tilde_W}
\widetilde W(x,y,t):=
\begin{cases}
W(x,y,t), & \text { if }y>0, \\\
W(x,-y,t), & \text { if }y<0,\\
\end{cases}
\end{equation*}
and let $\widetilde{\Phi_\la}$ be as in \eqref{def_tilde_Phi}.
Furthermore, let $\varphi \in C^\infty_c(\R^{N+1})$ be a cut-off function such that $\varphi= 1$ in $B_{r}$ and  $\varphi=0$  in $\R^{N+1}\setminus B_{2r}$.
Then
\begin{equation*}
\varphi(x,y) \widetilde W(x,y,t)\ge  C(r,t_1,t_2,w) \, \varphi(x,y) \widetilde\Phi_\la(x,y) \quad  \text{ a.e. in } \R^{N+1} \times(t_1,t_2)
\end{equation*}
and so
\begin{equation*}
 \rho_\e \ast (\varphi \widetilde W )\ge C(r,t_1,t_2,w) \, \rho_\e \ast (\varphi \widetilde{\Phi_\la}), \quad  \text{ in } \R^{N+1} \times(t_1,t_2),
\end{equation*}
where  $\{\rho_\e\}_{\e >0}$ is a family of standard mollifier on $\R^{N+1}$.

It  is clear that $\varphi \widetilde W \in H^1(\R^{N+1},|y|^{1-2s})$ and also $\varphi \widetilde{\Phi_\la}  \in H^1(\R^{N+1},|y|^{1-2s})$, in view of Proposition \ref{prop_Phi}. Then
by  \cite[Lemma 1.5]{K_weighted_Sobolev}, $\rho_\e   \ast (\phi \widetilde W ) \to \varphi \widetilde W$ and $\rho_\e \ast (\phi  \widetilde{\Phi_\la}) \to \phi  \widetilde{\Phi_\la}$
strongly  in $ H^1(\R^{N+1},|y|^{1-2s})$ as $ \e \to 0^+$. Hence, up to a subsequence,
\begin{equation*}
\Tr(\varphi W))(x,t)= \lim_{\e \to 0^+}(\rho_\e \ast (\phi \widetilde W))(x,0,t) \ge  \lim_{\e \to 0^+}  C(r,t_1,t_2,w)  (\rho_\e \ast (\varphi \Phi_\la))(x,0)
=C(r,t_1,t_2,w)\Tr(\varphi\Phi_\la))
\end{equation*}
a.e. in $\R^{N} \times(t_1,t_2)$.
It follows that
\begin{equation*}
w(x,t)\ge  C(r,t_1,t_2,w)|x|^{-\mu(\la)} \quad \text{ in } B'_{r} \times (t_1,t_2),
\end{equation*}
since
\begin{equation}
\Tr(\varphi W)(x,0,t)=\varphi(x,0) w(x,t) \quad \text{ and }\quad \Tr(\varphi \Phi_\la)(x,0)=\varphi(x,0) |x|^{-\mu(\la)}.
\end{equation}
Hence we have proved \eqref{ineq_singularity}.
\end{proof}

\section{Non-existence result and Fujita-type behaviour}\label{section_main_results}

In this section we analyse the existence and behaviour of solutions  to problem \eqref{prob_frac_heat_hardy} according to the value of the parameter $p$.

\subsection{A non-existence result}\label{subsection_non_existence}
In order to prove Theorem \ref{theor_non_existence}, we need a version of the Picone-type inequality for the extended problem.
\begin{proposition}\label{prop_Picone}
Let $r>0$ and assume $W \in H^1(B_r^+,y^{1-2s})$ is such that $W>\delta$ for some positive constant $\delta>0$ and $W$ weakly solves the problem
\begin{equation}\label{ellip}
\begin{cases}
-\dive(y^{1-2s}\nabla W)=y^{1-2s}g, &\text{ in } B_r^+,\\
-\lim\limits_{y \to 0^+}y^{1-2s}\pd{W}{y}=  f(x), &\text{ on } B_r',
\end{cases}
\end{equation}
in the sense that for any $\phi \in C^\infty_c(B_r^+\cup B_r')$
\begin{equation}\label{eq_Picone_ineq}
\int_{B_r^+} y^{1-2s} \nabla W \cdot \nabla \phi \, dz = \int_{B_r^+} y^{1-2s} g \phi \, dz+ \int_{B_r'}f \phi(x,0) \, dx,
\end{equation}
where $g \in L^2(B_r^+,y^{1-2s})$,  $f \in  L^1_{loc}(B_r') $  with $f \ge 0$, $f \neq 0$. Then for any $\phi\in C_c^\infty(B_r^+\cup B_r')$
\begin{equation}\label{ineq_Picone}
\int_{B_r^+}y^{1-2s}  |\nabla \phi|^2  \, dz \ge\int_{B_r'}\frac{f(x)}{\Tr(W)}\phi^2(x,0)dx+\int_{B_r^+} y^{1-2s} \frac{ g}{W}\phi^2 \, dz.
\end{equation}

\end{proposition}
\begin{proof}
Let $\phi\in C_c^\infty(B_r^+\cup B_r')$. Then
\begin{equation}
y^{1-2s}|\n \phi(x,y)|^2\ge y^{1-2s}\n \left(\frac{\phi^2(x,y)}{W(x,y)}\right)\n W(x,y)\quad \text{ a.e. in }B_r^+.
\end{equation}
Hence by integration, since $\phi^2 W^{-1} \in H^1(B_r^+,y^{1-2s})$, we obtain
\begin{multline}
\dyle \int_{B_r^+}y^{1-2s} |\n \phi|^2\, dz\\
\ge \int_{B_r^+}y^{1-2s}\n \left(\frac{\phi^2(x,y)}{W(x,y)}\right)\n W(x,y) \,dz =\int_{B_r^+} y^{1-2s} g \frac{\phi^2(x,y)}{W(x,y)} \, dz +\int_{B_r'}f \frac{\phi^2(x,0)}{\Tr(W)} \, dx
\end{multline}
which proves \eqref{ineq_Picone}.
\end{proof}

We have the following non-existence theorem for weak supersolutions.
\begin{theorem}\label{theor_non_existence}
Assume that $p>1+\frac{2s}{{\mu(\la)}}$. Let $f \in L^2(\R^{N+1})$
be such that $f(x,t)=0$ a.e. in $\R^N\times (-\infty,0)$. Then any
non-negative weak supersolution  $w$ of problem
\eqref{prob_frac_heat_hardy} is trivial. In particular if $f\not
\equiv 0$, there are not any   non-negative weak supersolutions.
\end{theorem}
\begin{proof}
We argue by contradiction. Let $p>1+\frac{2s}{{\mu(\la)}}$ and
suppose that \eqref{eq_frac_heat_hardy_weak} has a non-negative
supersolution $w$. Then there exists a solution of
\eqref{eq_frac_heat_hardy_weak} and a sequence $\{w_n\}_{n
\in\mathbb{N}}$ as in Proposition
\ref{prop_existence_weak_solution}. For any  $r>0$  and any
$(t_1,t_2) \subset (0,+\infty)$ by standard abstract parabolic
elliptic theory, see for example \cite{E_book_pdes}, letting $W_n$
be the extension of $w_n$ as in Theorem \ref{theorem_extension},
it follows that $(W_n)_t \in L^2((t_1,t_2),B_r^+)$. Fix $r>0$ and
$(t_1,t_2) \subset (0,+\infty)$. Then for any  $\phi \in
C_c^{\infty}(B_r^+ \cup B_r')$ by  \eqref{prob_frac_heat_hardy},
Proposition \ref{prop_Picone}, Theorem \ref{theorem_extension},
and the regularity of $(W_n)_t$ it follows  that
\begin{equation}
\int_{B^+_r} y^{1-2s}|\n \phi|^2  dz  \ge \kappa_s  \int_{B'_r}\frac{w^p_{n-1}}{(1+\frac{1}{n}w^p_{n-1})w_n}\phi^2(x,0)dx + \int_{B^+_r}y^{1-2s}\phi^2\frac{(W_n)_t}{W_n}dz.
\end{equation}
Furthermore
$\{w_n\}_{n \in \mathbb{N}}$ is bounded in $\mathop{{\rm{Dom}}(H^s)}$ and so for any $\sigma>1$   there exits a constant $C$, that does not depends on $n$, such that for any $n \in \mathbb{N}$
\begin{equation}
\int_{B^+_r}y^{1-2s}|\log W_n(x,y,t_i)|^{\sigma}dz \le C
\end{equation}
for  $i=1,2$, see \cite{BCS}.
It is also clear that
\begin{equation}
\int_{t_1}^{t_2}\int_{B^+_r}y^{1-2s}\phi^2\frac{(W_n)_t}{W_n}dx=\int_{B^+_r}y^{1-2s}\phi^2(\log(W_n(z,t_1))- \log(W_n(z,t_2))dz.
\end{equation}
Then, letting
\begin{equation}
2^{**}_s:=\min\left\{\frac{N+2-2s}{N-2s},2\frac{N+1}{N-1}\right\}
\end{equation}
and $\sigma$ its conjugate exponent, that is $\frac{1}{2^{**}_s} +\frac{1}{\sigma}=1$,  by the H\"older inequality and \cite[Lemma 4.2]{FS_regularity},
\begin{multline}
\int_{B_r^+}y^{1-2s}\phi^2|\log W_n(x,y,t_i)|dz\\
\le \bigg(\int_{B_r^+}y^{1-2s}|\phi|^{2^{**}_s}dz\bigg)^{\frac{2}{2^{**}_s}}\bigg(\int_{B^+_r}y^{1-2s}|\log W_n(x,y,t_i)|^{\sigma}dz\bigg)^{\frac{1}{\sigma}}\\
\le  C^\frac{1}{\sigma} \int_{B^+_r} y^{1-2s} |\n \phi|^2\, dz,
\end{multline}
for $i=1,2$. Hence, combining the above estimates and integrating over $(t_1,t_2)$, we obtain that there exits a constant $C_1>0$, that does not depend on $\phi$, such that
\begin{equation}
(t_2-t_1)\int_{B^+_r} y^{1-2s}|\n \phi|^2  dz  \ge C_1 \int_{t_1}^{t_2}\int_{B'_r}\frac{w^p_{n-1}}{(1+\frac{1}{n}w^p_{n-1})w_n}\phi^2(x,0)dx \, dt.
\end{equation}
Passing to the limit as $n \to \infty$ by the Dominated Converge Theorem,  we conclude that
\begin{equation}
\int_{B^+_r} y^{1-2s} |\n \phi|^2dz \ge C_1\int_{B'_r}|x|^{-(p-1){\mu(\la)}}\phi^2(x,0)dx,
\end{equation}
in view of Proposition \ref{prop_estimate_singularity}.
Since $(p-1){\mu(\la)}>2s$, then we have  reached a contradiction with the optimality of the power $|x|^{2s}$ in the Hardy-type inequality \eqref{ineq_Hardy_frac}.
\end{proof}
\begin{remark}\label{reamrk_local}
If $p<p_+(\la,s)$,  we get the existence of a local
pointwise positive supersolution to \eqref{prob_frac_heat_hardy}
in any bounded domain $\Omega \times (T_1,T_2)$. Indeed using
\cite[Corollary 1.4.]{ST}, there exists positive supersolution to
the correspondent elliptic problem
\begin{equation*}
\begin{cases}
(-\Delta )^s\psi=\frac{\la}{|x|^{2s}} \psi+\psi^p +h, &\text{  in  }\Omega,\\
    \psi(x)=0, &\text{  in  }\R^N\setminus \Omega,
\end{cases}
\end{equation*}
where $h$, depending only on $x$, satisfies some additional
assumptions. 
\end{remark}

\subsection{A blow-up result}\label{subsection_blow_up}
In  the case $1<p\le1+\frac{2s}{N+2-2s-\mu(\lambda)}$, we have the
following blow-up result.
\begin{theorem}\label{theor_Fujita}
Assume that $1 < p \le   1+\frac{2s}{N+2-2s-{\mu(\lambda)}}$ and that
$w$ is a non-negative, non-trivial  solution of \eqref{prob_frac_heat_hardy} either in the sense of Definition \ref{def_very_weak_solutions} or in the sense of Definition \ref{def_weak_solutions}. Then there exists $T^*<\infty$ such that
\begin{equation}\label{PPP}
\lim\limits_{t\to
T^{*}}\int_{\R^N}|x|^{-\mu(\lambda)}w^p(x,t)\,dx=\iy.
\end{equation}
\end{theorem}

\begin{proof}
Let $w$ be a positive non-trivial   solution of \eqref{prob_frac_heat_hardy} in the sense of Definition \ref{def_very_weak_solutions} or in the sense of Definition \ref{def_weak_solutions}.
Let us define
\begin{equation}\label{def_v}
v(x,t):=|x|^{\mu(\la)}w(x,t),
\end{equation}
so that  $|x|^{-\mu(\la)}w^p(x,t)=|x|^{-(p+1)\mu(\la)}v^p(x,t)$. We argue by contradiction assuming that
\begin{equation*}
\limsup_{t\to T}\int_{\R^N}|x|^{-\mu(\lambda)}w^p(x,t)\,dx<+\infty  \quad \text{ for all  }T\ge0.
\end{equation*}
Then  in particular
\begin{equation*}
\int_0^{T}\int_{\R^N}|x|^{-(p+1)\mu(\la)}v^p(x,t) \,dx dt<\infty  \quad \text{ for all  }T>0.
\end{equation*}
If $w$ is a solution of   \eqref{prob_frac_heat_hardy} in the sense of Definition \ref{def_very_weak_solutions} then for any even $\phi \in C^{\infty}_c(\R^{N+1})$
\begin{multline*}
\int_{\R^{N+1}} w(x,t) H^s(\phi)(-x,-t) \, dx \, dt = \int_{\R^{N+1}}J_s\left(\frac{\la}{|x|^{2s}}w +w^p+f\right)H^s(\phi)(-x,-t)\, dx \, dt \\
=\int_{\R^{N+1}}\left(\frac{\la}{|x|^{2s}}w +w^p+f\right)\phi(x,t)\, dx \, dt
\end{multline*}
by Lemma \ref{lemma_H_j_adjoint}  and a change of variables. If
$w$ is a solution of   \eqref{prob_frac_heat_hardy} in the sense
of Definition \ref{def_very_weak_solutions} then the equation
above still holds in view of Lemma \ref{lemma_H_j_adjoint}  and
\eqref{eq_frac_heat_hardy_weak}. Then, by Proposition
\ref{prop_ground_state},  in both cases we have that
\begin{equation}\label{proof:theor_Fujita:3}
\int_{\R^{N+1}}v(x,t)^p |x|^{-\mu(\la)(p+1)}\phi(x,t)\, dx dt\le \int_{\R^{N+1}} v(-x,-t) |x|^{-\mu(\la)}L^s(\phi)\, dx dt
\end{equation}
for any even  $\phi \in C^{\infty}_c(\R^{N+1})$.

Let $\varphi\in C^{\infty}_c(B_2')$ be an even cut-off function such that $0\le\varphi\le 1$, $\varphi\equiv 1$ in $B_1'$.  For any $R \ge 1$ let us define
\begin{equation}\label{def_varphi_R}
\varphi_R(x,t):=\varphi\left(\frac{x}{R},\frac{t}{R^2}\right).
\end{equation}
Fix $m > p'$. Then testing \eqref{proof:theor_Fujita:3}  with $\varphi_R^m$ we obtain
\begin{multline}\label{proof:theor_Fujita:4}
\int_{\R^{N+1}}v(x,t)^p |x|^{-\mu(\la)(p+1)}\varphi^m_R(x,t)\, dx dt\le\int_{\R^{N+1}} v(-x,-t) |x|^{-\mu(\la)}L^s(\varphi^m_R)(x,t)\, dx dt \\
\le m \int_{\R^{N+1}} v(-x,-t) |x|^{-\mu(\la)}\vp^{m-1}_R(x,t)L^s(\varphi_R)(x,t)\, dx dt,
\end{multline}
in view of Proposition \ref{prop_kato}. Letting $p'$ be such that $1/p+1/p'=1$, from the H\"older inequality it follows that
\begin{multline}\label{proof:theor_Fujita:5}
\int_{\R^{N+1}}v(x,t)^p |x|^{-\mu(\la)(p+1)}\varphi_R^m(x,t)\, dx dt
\le m \left(\int_{\R^{N+1}}v(-x,-t)^p |x|^{-\mu(\la)(p+1)}\varphi^m_R(x,t)\, dx dt\right)^\frac{1}{p}\\
\times\left(\int_{\R^{N+1}}  |x|^{\mu(\la)(p'-1)}\varphi_R^{m-p'}(x,t)(L^s(\varphi_R)(x,t))^{p'}\, dx dt\right)^\frac{1}{p'}.
\end{multline}
Then by a change of variables, we deduce that
\begin{multline}\label{proof:theor_Fujita:5.1}
\int_{\R^{N+1}}v(x,t)^p |x|^{-\mu(\la)(p+1)}\varphi^m_R(x,t)\, dx dt\\
\le \int_{\R^{N+1}}  |x|^{\mu(\la)(p'-1)}\varphi_R^{m-p'}(x,t)(L^s(\varphi_R)(x,t))^{p'}\, dx dt.
\end{multline}
Furthermore the change of variables $\tilde{y}=R^{-1} y$ and $\tilde{\sigma}=R^{-2} \sigma$ yields
\begin{multline}\label{proof:theor_Fujita:6}
L^s(\varphi_R)(x,t)=\frac{1}{(4\pi)^{\frac{N}{2}}\Gamma(-s)}
\int_{-\infty}^t\int_{\R^N}\frac{\varphi_R(x,t)-\varphi_R(y,\sigma)}{|y|^{{\mu(\la)}}(t-\sigma)^{\frac{N}{2}+s+1}}e^{-\frac{|x-y|^2}{4(t-\sigma)}}dyd\sigma\\
=\frac{1}{(4\pi)^{\frac{N}{2}}\Gamma(-s)}R^{-\mu(\la)-2s}
\int_{-\infty}^{\frac{t}{R^2}}\int_{\R^N}\frac{\varphi(R^{-1}x,R^{-2}t)-\varphi(\tilde{y},\tilde{\sigma})}{|\tilde{y}|^{{\mu(\la)}}(R^{-2}t-\tilde \sigma)^{\frac{N}{2}+s+1}}
e^{-\frac{|R^{-1}x-\tilde{y}|^2}{4(R^{-2}t-\tilde{\sigma})}}d\tilde{y}d\tilde{\sigma}\\
 \le C_1 R^{-2s} |x|^{-\mu(\la)},
\end{multline}
for some positive constant $C_1>0$ depending only on $N,s, \la $ and $\varphi$, by Proposition \ref{prop_ground_state}.
Hence
\begin{equation}\label{proof:theor_Fujita:7}
\int_{\R^{N+1}}v(x,t)^p |x|^{-\mu(\la)(p+1)}\varphi^m_R(x,t)\, dx dt\\
\le C_2  R^{-2sp'}\int_{\R^{N+1}}  |x|^{-\mu(\la)}\varphi_R^{m-p'}(x,t)\, dx dt,
\end{equation}
for some positive constant $C_2>0$.
The change of variables $\tilde{x}=R^{-1} x$ and $\tilde{t}=R^{-2} t$ yields
\begin{multline}\label{proof:theor_Fujita:8}
\int_{\R^{N+1}}v(x,t)^p |x|^{-\mu(\la)(p+1)}\varphi_R^m(x,t)\, dx dt\\
\le C_2  R^{N+2-\mu(\la)-2sp'}\int_{\R^{N+1}} |x|^{-\mu(\la)}\varphi^{m-p'}(\tilde{x},\tilde{t})\, d\tilde{x} d\tilde{t}.
\end{multline}
Thanks to Fatou's Lemma, we can pass to the limit as $R \to \infty$  and conclude that
\begin{equation}\label{proof:theor_Fujita:9}
\int_{\R^{N+1}}v(x,t)^p |x|^{-\mu(\la)(p+1)}\, dx dt <+\infty.
\end{equation}
In particular, if $p<1+\frac{2s}{N+2-2s-\mu(\la)}$,  we obtain
\begin{equation}
\int_{\R^{N+1}}v(x,t)^p |x|^{-\mu(\la)(p+1)}\, dx dt=0,
\end{equation}
thus  $v\equiv 0$.  Since $w(x,t)=v(x,t)|x|^{-\mu(\la)}$, then $w
\equiv 0$ in $\R^N\times \R$, a contradiction.

We deal now with the critical case $p=1+\frac{2s}{N+2-2s-\mu(\la)}$. Let $\delta \in (0,\frac{1}{p'-1})$. Then
\begin{multline}\label{proof:theor_Fujita:10}
\int_{\R^{N+1}} v(x,t) |x|^{-\mu(\la)}\varphi^{m-1}_R(x,t)L^s(\varphi_R(x,t))\, dx dt \\
=\int_{\R^{N+1}} v(x,t) |x|^{-\mu(\la)}\varphi^{m-1}_R(x,t)(1-\varphi_R(x,t))^\delta(1-\varphi_R(x,t))^{-\delta} L^s(\varphi_R(x,t))\, dx dt\\
\le\left( \int_{\R^{N+1}}  |x|^{\mu(\la)(p'-1)}\varphi_R^{m-p'}(x,t)(1-\varphi_R(x,t))^{\delta(1-p')}(L^s(\varphi_R(x,t)))^{p'}\, dx dt\right)^\frac{1}{p'}\\
\times \left(\int_{\R^{N+1}}v(x,t)^p |x|^{-\mu(\la)(p+1)}\varphi^m_R(x,t)(1-\varphi_R(x,t))^{\delta}\, dx dt\right)^\frac{1}{p}.
\end{multline}
Furthermore, arguing as above, we can show that
\begin{multline}\label{proof:theor_Fujita:11}
\int_{\R^{N+1}}  |x|^{\mu(\la)(p'-1)}\varphi_R^{m-p'}(x,t)(1-\varphi_R(x,t))^{\delta(1-p')}(L^s(\varphi_R(x,t)))^{p'}\, dx dt \\
\le  C_2 \int_{\R^{N+1}}  |x|^{-\mu(\la)}\varphi^{m-p'}(x,t)(1-\varphi(x,t))^{\delta(1-p')}\, dx dt.
\end{multline}
Since we have chosen   $\delta \in (0,\frac{1}{p'-1})$, we conclude that the  integral in the right hand side of \eqref{proof:theor_Fujita:11}  is finite.
In conclusion from \eqref{proof:theor_Fujita:10} and \eqref{proof:theor_Fujita:11} we deduce that there exists a constant $C_3>0$ such that
\begin{equation}
\int_{\R^{N+1}} v(x,t) |x|^{-\mu(\la)}\varphi^{m-1}_R(x,t)L^s(\varphi_R(x,t))\, dx dt
\le C_3\left(\int_{\R^{N+1}\setminus B'_R}v(x,t)^p |x|^{-\mu(\la)(p+1)}\, dx dt\right)^\frac{1}{p}
\end{equation}
for any $R\ge 1$. Passing to the limit as $R \to \infty$   we conclude that
\begin{equation}
\lim_{R\to \infty}\int_{\R^{N+1}} v(x,t) |x|^{-\mu(\la)}\varphi^{m-1}_R(x,t)L^s(\varphi_R(x,t))\, dx dt =0.
\end{equation}
Then from \eqref{proof:theor_Fujita:5.1} and  Fatou's Lemma we deduce that  $v\equiv 0$ and so $w\equiv0$.
\end{proof}

\begin{remark}\label{remark_la=0}
It is easy too see that, if $\la=0$, then proceeding as in Theorem \ref{theor_Fujita}, we can show that
there exists $T^*<\infty$ such that
\begin{equation*}   \lim\limits_{t\to   T^{*}}\int_{\R^N}w^p(x,t)\,dx=\iy,
\end{equation*}
see also \cite{T_frac_heat}.
\end{remark}

\subsection{Existence of supersolutions}\label{subsection_existence_super}
We start by proving the following result.
\begin{proposition}\label{prop_super_sol}
Let $p \in \left(1+\frac{2s}{N-\mu(\la)+2-2s}, 1+\frac{2s}{\mu(\la)}\right)$. Then there exist $\e_0>0$ and  $\delta>0$ such that for any $\e \in (0,\e_0]$ and any $\la_1 \in (\la,\la+\delta]$ the function
\begin{equation}\label{def_U}
U_{\e, \la_1}(x,y,t):= \e (1+t)^{\frac{\mu(\la_1)}{2}-\frac{s}{p-1}} \Phi_{\la_1}(z) e^{-\frac{|z|^2}{4(t+1)}}
\end{equation}
is a positive classical solution of the problem
\begin{equation}\label{prob_Ue}
\begin{cases}
y^{1-2s}(U_{\e, \la_1})_t-\dive(y^{1-2s}\nabla U_{\e, \la_1})\ge 0, &\text{ in } \R_+^{N+1}\setminus\{0\} \times (0,+\infty),\\
-\lim\limits_{y \to 0^+}y^{1-2s}\pd{U_{\e, \la_1}}{y}\ge\kappa_s (\frac{\la}{|x|^{2s}}u_{\e, \la_1}+u_{\e, \la_1}^p), &\text{ on }(\R^N\setminus\{0\}) \times (0,+\infty),\\
U_{\e, \la_1}(z,0)\ge 0, &\text{ on }\R_+^{N+1}\setminus\{0\},
\end{cases}
\end{equation}
where $u_{\e, \la_1}(x,t):=U_{\e, \la_1}(x,0,t)=\e
(1+t)^{\frac{s}{p-1}-\frac{\mu(\la_1)}{2}} |x|^{-\mu(\la_1)}
e^{-\frac{|x|^2}{4(t+1)}}$ and $k_s$ is as in Theorem
\ref{theorem_extension}. Furthermore for any $T>0$, we have that
$U_{\e, \la_1} \in L^2((0,T),H^1(\R^{N+1}_+,y^{1-2s}))$ and
$(U_{\e, \la_1})_t \in L^2((0,T),L^2(\R^{N+1}_+,y^{1-2s})$ .
\end{proposition}
\begin{proof}
Let us consider the family of functions
\begin{equation*}
U(x,y,t):=\e (1+t)^{-\theta} \Phi_{\la_1}(z) e^{-\frac{|z|^2}{4(t+1)}},
\end{equation*}
with $\e >0$, $\theta>0$ and $\la_1 \in (\la,\Lambda_{N,s})$.
With a direct computation we can see that
\begin{align*}
&U_t(z,t)=\e\left[-\theta +\frac{1}{4}t^{-1}|z|^2\right]    (1+t)^{-\theta-1} \Phi_{\la_1}(z) e^{-\frac{|z|^2}{4(1+t)}}, \\
&\dive(y^{1-2s}\nabla U)(z,t)= \e y^{1-2s}  \left[-\mu(\la_1)-\frac{1}{2}(N+2-2s) +\frac{1}{4}t^{-1}|z|^2\right]    (1+t)^{-\theta-1} \Phi_{\la_1}(z) e^{-\frac{|z|^2}{4(1+t)}},
\end{align*}
thanks to \eqref{prob_Phi} and \eqref{eq_Phi_nabla}. Hence
\begin{multline}\label{proof_super_sol:1}
y^{1-2s}U_t(z,t)-\dive(y^{1-2s}\nabla U)(z,t)\\
=\e y^{1-2s}  \left[-\theta -\mu(\la_1)+\frac{1}{2}(N+2-2s)\right]  (1+t)^{-\theta-1} \Phi_{\la_1}(z) e^{-\frac{|z|^2}{4(1+t)}}.
\end{multline}
On the other hand by  \eqref{prob_Phi}
\begin{equation*}
-\lim_{y \to 0^+}y^{1-2s} \pd{U}{y}= \e (1+t)^{-\theta} \kappa_s\frac{\la_1}{|x|^{2s}}|x|^{-\mu(\la_1)}e^{-\frac{|x|^2}{4(1+t)}}.
\end{equation*}
Then $-\lim\limits_{y \to 0^+}y^{1-2s}\pd{U_{\e, \la_1}}{y}\ge k_s \left(\frac{\la}{|x|^{2s}}u_{\e, \la_1}+u_{\e, \la_1}^p\right)$ if and only if
\begin{equation*}
 (1+t)^{-\theta} (\la_1-\la)|x|^{-\mu(\la_1)-2s} e^{-\frac{|x|^2}{4(1+t)}} \ge \e^{p-1} (1+t)^{-\theta p} |x|^{-p\mu(\la_1)} e^{-p\frac{|x|^2}{4(1+t)}}.
\end{equation*}
The change of variables $x:=(1+t)^{\frac{1}{2}}\xi$ yields
\begin{equation*}
(1+t)^{-\theta-\frac{\mu(\la_1)}{2}-s} (\la_1-\la)|\xi|^{-\mu(\la_1)-2s} \ge \e^{p-1} (1+t)^{-\theta p-\frac{p\mu(\la_1)}{2}} |\xi|^{-p\mu(\la_1)} e^{-(p-1)\frac{|\xi|^2}{4}}.
\end{equation*}
Choosing $\theta:=\frac{s}{p-1}-\frac{\mu(\la_1)}{2}$ we obtain
\begin{equation}\label{proof_super_sol:2}
(\la_1-\la)|\xi|^{p\mu(\la_1)-\mu(\la_1)-2s} \ge \e^{p-1}  e^{-(p-1)\frac{|\xi|^2}{4}}.
\end{equation}
With this choice of $\theta$ and a direct computation we can see that
\begin{equation}
-\theta -\mu(\la_1)+\frac{1}{2}(N+2-2s) >0
\end{equation}
if and only if
\begin{equation}
 p> 1+\frac{2s}{N+2-2s-\mu(\la_1)}.
\end{equation}
Hence, choosing $\la_1$ close enough to $\la$,
\begin{equation}
y^{1-2s}U_t(z,t)-\dive(y^{1-2s}\nabla U)(z,t) \ge 0.
\end{equation}
Finally, up to choosing  $\la_1$  closer to $\la$, we also have  that
\begin{equation}
p< 1+\frac{2s}{\mu(\la_1)}
\end{equation}
and so \eqref{proof_super_sol:2} holds for any $\xi \in \R^{N}$ if $\e>0$ is small enough.
In conclusion we have proved that $U_{\e,\la_1}$ is a classical solution of  problem \eqref{prob_Ue}. Since $\Phi_{\la_1} \in H^1(B^+_1,y^{1-2s})$, the last claim is clear.
\end{proof}

\begin{theorem}\label{theor_exsitence_super}
Let $p \in \left(1+\frac{2s}{N-\mu(\la)+2-2s}, 1+\frac{2s}{\mu(\la)}\right)$ and
$\delta,\e_0 $ as in Proposition \ref{prop_super_sol}.
Let $f$ be a measurable non-negative, non-trivial function  such that $f(x,t)=0$ a.e. in $\R^N
\times (-\infty,0)$. Assume that for some $\delta_1 \in (0,\delta)$
\begin{equation}\label{ineq_f}
f(x,t) \le \delta_1 (1+t)^{\frac{\mu(\la_1)}{2}-\frac{s}{p-1}} |x|^{-\mu(\la_1)-2s} e^{-\frac{|x|^2}{4(1+t)}}.
\end{equation}
Then for some a small enough  $\e \in (0,\e_0)$ and $\la_1 \in (\la,\la+\delta)$ close enough to $\la$, the function
\begin{multline}\label{eq_w_super_sol_rappresentation}
w(x,t):=\frac{\Gamma(1-s)}{2^{2s-1}\Gamma(s)^2(4\pi)^{\frac{N}{2}}}\int_0^\infty\int_{\R^N}  \frac{e^{-\frac{|z|^2}{4\tau}}}{\tau^{\frac{N}{2}+1-s}} \, \,
\Big[\e \la(1+t-\tau)^{\frac{s}{p-1}-\frac{\mu(\la_1)}{2}} |x-z|^{-\mu(\la_1)-2s} e^{-\frac{|x-z|^2}{4(t-\tau+1)}}\\
+\e^p(1+t-\tau)^{\frac{ps}{p-1}-\frac{p\mu(\la_1)}{2}} |x-z|^{-p\mu(\la_1)} e^{-p\frac{|x-z|^2}{4(t-\tau+1)}}
+f(x-z,t-\tau)\Big]\, dz \, d\tau,
\end{multline}
 is a non-trivial, non-negative,  very weak supersolution of   problem
\eqref{prob_frac_heat_hardy} and   $w\in L^1_{loc}(\R^{N+1})$.
\end{theorem}
\begin{proof}
We start by observing that  for any $p \in \left(1+\frac{2s}{N-\mu(\la)+2-2s}, 1+\frac{2s}{\mu(\la)}\right)$ and $\la_1 \in (\la,\la+\delta)$ the function
\begin{equation}\label{def_varphi}
\varphi(t,x)=(1+t)^{(p-1)\frac{\mu(\la_1)}{2}-s}|x|^{-p\mu(\la_1)+\mu(\la_1)+2s}e^{-(p-1)\frac{|x|^2}{4(t+1)}}
\end{equation}
is bounded on $\R^N \times (0,+\infty)$.
It follows that, if $f$ is as above, taking for example $\la_1:= \la +\frac{\delta_1}{2}$, there exists an  $\e>0$ such that the function $U_{\e,\la_1}$ is a classical supersolution of the problem
\begin{equation}
\begin{cases}
y^{1-2s}(U_{\e, \la_1})_t-\dive(y^{1-2s}\nabla U_{\e, \la_1})\ge 0, &\text{ in } \R_+^{N+1}\setminus\{0\} \times (0,+\infty),\\
-\lim\limits_{y \to 0^+}y^{1-2s}\pd{U_{\e, \la_1}}{y}\ge k_s(\frac{\la}{|x|^{2s}}u_{\e, \la_1}+u_{\e, \la_1}^p+f), &\text{ on }(\R^N\setminus\{0\}) \times (0,+\infty),\\
U_{\e, \la_1}(z,0)\ge 0, &\text{ on }\R_+^{N+1}\setminus\{0\}.
\end{cases}
\end{equation}
Let us extend trivially the function $u_{\e, \la_1}$ to $\R^{N+1}$ and still denote, with a slight abuse of notation, the extended function with  $u_{\e, \la_1}$. Then
$\frac{\la}{|x|^{2s}}u_{\e, \la_1}+u_{\e, \la_1}^p+f \in L^1(\R^{N}\times (-\infty,T))$ for any $T>0$ and $\frac{\la}{|x|^{2s}}u_{\e, \la_1}+u_{\e, \la_1}^p+f\equiv 0$ on $\R^N \times(-\infty,0)$.
It follows that, letting $k_s$ be as in Theorem \ref{theorem_extension},
\begin{equation}\label{def_super_sol_w}
w:=k_sJ_s\left(\frac{\la}{|x|^{2s}}u_{\e, \la_1}+u_{\e, \la_1}^p+f\right)
\end{equation}
is well-defined, $w \in L^1_{loc}(\R^{N+1})$ and $w \equiv 0$ on $\R^N \times(-\infty,0)$ by \eqref{Js_properties_intial_datum}.
Let $\eta_n$ be as in Subsection \ref{subsection_existence} and let us define for any $n \in \mathbb{N}$
\begin{equation}\label{def_wn}
w_n:=k_sJ_s\left(\eta_n\left[\frac{\la}{(|x|+\frac{1}{n})^{2s}}\frac{u_{\e, \la_1}}{1+\frac{1}{n}u_{\e, \la_1}}
 +\frac{u^p_{\e, \la_1}}{1+\frac{1}{n}u^p_{\e, \la_1}}+\frac{f}{1+\frac{1}{n}f}\right] \right).
\end{equation}
Then  $w_n \in \mathop{\rm{Dom}}(H^s)$ thanks to Lemmas \ref{lemma_Jsg_in L^2} and  \ref{lemma_Jsg_in DHs}. Hence we may define  its extension $W_n$ as in Theorem \ref{theorem_extension} and notice
that $W_n$ weekly  solves the problem
\begin{equation*}
\begin{cases}
y^{1-2s}(W_n)_t-\dive(y^{1-2s}\nabla W_n)=0, &\text{ in } \R_+^{N+1} \times (0,+\infty),\\
-\lim\limits_{y \to 0^+}y^{1-2s}\pd{W_n}{y}= \kappa_s \eta_n\left[\frac{\la}{(|x|+\frac{1}{n})^{2s}}\frac{u_{\e, \la_1}}{1+\frac{1}{n}u_{\e, \la_1}}
+\frac{u^p_{\e, \la_1}}{1+\frac{1}{n}u^p_{\e, \la_1}}+\frac{f}{1+\frac{1}{n}f}\right], &\text{on }\R^N\times \times (0,+\infty),\\
W_n(z,0)=0 &\text{ on }\R^{N+1}_+.
\end{cases}
\end{equation*}
By Corollary \ref{corollary_comparison_extended}, it follows that $w_n \le w_{n+1}\le u_{\e,\la_1}$ for any $n \in \mathbb{N}$. Then, by the Dominated Convergence Theorem,
\begin{multline*}
w=k_sJ_s\left(\frac{\la}{|x|^{2s}}u_{\e, \la_1}+u_{\e, \la_1}^p+f\right)\\
=\lim_{n\to \infty}k_sJ_s\left(\eta_n\left(\frac{\la}{(|x|+\frac{1}{n})^{2s}}\frac{u_{\e, \la_1}}{1+\frac{1}{n}u_{\e, \la_1}}
+\frac{u^p_{\e, \la_1}}{1+\frac{1}{n}u^p_{\e, \la_1}}+\frac{f}{1+\frac{1}{n}f}\right) \right)=\lim_{n\to \infty}w_n \le u_{\e,\la_1},
\end{multline*}
a.e. in $\R^{N+1}$. We conclude that
\begin{equation*}
w=k_sJ_s\left(\frac{\la}{|x|^{2s}}u_{\e, \la_1}+u_{\e, \la_1}^p+f\right) \ge J_s\left(\frac{\la}{|x|^{2s}}w+w^p+f\right).
\end{equation*}
Hence $w$ is a very  weak supersolution of \eqref{prob_frac_heat_hardy}, that is, is a supersolution  in the sense of Definition \ref{def_very_weak_solutions}.
\end{proof}

From Theorem \ref{theor_exsitence_super} and Proposition \ref{prop_existence_very_weak_solution} we can immediately deduce the existence of very weak solutions of problem \eqref{prob_frac_heat_hardy}
for small datum.
\begin{corollary}\label{corollary_existence_very_weak_solution}
Let $p \in \left(1+\frac{2s}{N-\mu(\la)+2-2s}, 1+\frac{2s}{\mu(\la)}\right)$ and suppose that $f$ satisfies the assumptions of Theorem \ref{theor_exsitence_super}.
Then  \eqref{prob_frac_heat_hardy} has a non-negative, non-trivial  very weak  solution $w\in L^1_{loc}(\R^{N+1})$.
\end{corollary}

\section{Open problems and subjects of further  investigation} \label{sec_open_problems}
In this last section we make a brief overview over remaining questions and further developments  concerning existence and non-existence results for the operator $H^s-\frac{\la}{|x|^{2s}}$ and other
fractional   parabolic operators.

As already stated in  Remark \ref{remark_inversion_Ls}, it is an
interesting open problem to obtain an inversion formula for the
operator $H^s-\frac{\la}{|x|^{2s}}$. Furthermore in the case
$p=p_+(\la,s)$  it is still not known if there exists or not weak
supersolution of \eqref{prob_frac_heat_hardy}. It is reasonable to
think that a non-existence result may hold, coherently with the
classical case $s=1$, but to obtain such a result seems to be
technically demanding.

Further possible subject of investigation includes considering
different non-linearity in problem \eqref{prob_frac_heat_hardy},
for example of the form $|\nabla u|^p$ or
$|(-\Delta)^{\frac{s}{2}}u|^p$, which, to the best of the authors
knowledge, have yet to be studied with or without the presence of
an Hardy-type potential. It may also be interesting to study
similar questions for  more general   parabolic fractional
operators, for example $(w_t- \Delta-\frac{\la}{|x|^2})^s$, under
assumption of positivity of the elliptic part  $-
\Delta-\frac{\la}{|x|^2}$. Finally the critical case
$\la=\Lambda_{N_s}$ is yet to be studied for the operator
$H^s-\frac{\la}{|x|^{2s}}$ and could be of interest.

All the subject mentioned above are object of current investigation by the authors.

\end{document}